\documentclass[reqno,11pt]{amsart}
\usepackage[margin=1in]{geometry}
\usepackage{graphicx}

\usepackage{amsthm, amsmath, amssymb, bm, bbm}
\usepackage{tikz}

\usepackage[utf8]{inputenc}
\usepackage[T1]{fontenc}

\usepackage[textsize=scriptsize,backgroundcolor=orange!5]{todonotes}

\usepackage{hyperref}
\usepackage{url}

\usepackage[noabbrev,capitalize]{cleveref}
\crefname{equation}{}{}

\usepackage{xcolor,graphics, graphicx}
\usepackage{floatrow}
\usepackage{verbatim}
\newfloatcommand{capbtabbox}{table}[][\FBwidth]

\numberwithin{equation}{section}

\newtheorem{theorem}{Theorem}[section]
\newtheorem{proposition}[theorem]{Proposition}
\newtheorem{lemma}[theorem]{Lemma}

\newtheorem{corollary}[theorem]{Corollary}
\newtheorem{conjecture}[theorem]{Conjecture}
\newtheorem*{question*}{Question}
\Crefname{question}{Question}{Questions}

\newtheorem{assumption}[theorem]{Assumption}

\theoremstyle{definition}
\newtheorem{definition}[theorem]{Definition}

\newtheorem{example}[theorem]{Example}
\newtheorem*{example*}{Example}

\theoremstyle{remark}
\newtheorem*{remark}{Remark}

\newcommand{\one}{\mathbbm{1}}
\newcommand{\wt}{\widetilde}

\DeclareMathOperator{\Sym}{Sym}
\DeclareMathOperator{\rad}{rad}
\DeclareMathOperator{\GL}{GL}

\renewcommand{\Re}{\operatorname{Re}}

\newcommand{\mb}{\mathbf}
\newcommand{\mc}{\mathcal}

\newcommand{\unr}{\text{unr}}

\renewcommand{\AA}{\mathbb{A}}

\newcommand{\FF}{\mathbb{F}}
\newcommand{\QQ}{\mathbb{Q}}
\newcommand{\RR}{\mathbb{R}}

\newcommand{\PP}{\mathcal{P}}
\newcommand{\NN}{\mathbb{N}}

\newcommand{\eps}{\varepsilon}
\renewcommand{\pmod}[1]{\ (\mathrm{mod}\ #1)}

\title{Patterns of primes in the Sato--Tate conjecture}

\author[Gillman]{Nate Gillman}
\address{Wesleyan University, Middletown, CT 06459, USA}
\email{ngillman@wesleyan.edu}

\author[Kural]{Michael Kural}
\address{Massachusetts Institute of Technology, Cambridge, MA 02139, USA}
\email{mkural@mit.edu}

\author[Pascadi]{Alexandru Pascadi}
\address{University of California, Los Angeles, CA 90095, USA}
\email{alexpascadi@ucla.edu}

\author[Peng]{Junyao Peng}
\address{Massachusetts Institute of Technology, Cambridge, MA 02139, USA}
\email{junyaop@mit.edu}

\author[Sah]{Ashwin Sah}
\address{Massachusetts Institute of Technology, Cambridge, MA 02139, USA}
\email{asah@mit.edu}

\date{\today}

\allowdisplaybreaks

\begin{document}

\maketitle

\begin{abstract}
Fix a non-CM elliptic curve $E/\mathbb{Q}$, and let $a_E(p) = p + 1 - \#E(\mathbb{F}_p)$ denote the trace of Frobenius at $p$. The Sato--Tate conjecture gives the limiting distribution $\mu_{ST}$ of $a_E(p)/(2\sqrt{p})$ within $[-1, 1]$. We establish bounded gaps for primes in the context of this distribution. More precisely, given an interval $I\subseteq [-1, 1]$, let $p_{I,n}$ denote the $n$th prime such that $a_E(p)/(2\sqrt{p})\in I$. We show $\liminf_{n\to\infty}(p_{I,n+m}-p_{I,n}) < \infty$ for all $m\ge 1$ for ``most'' intervals, and in particular, for all $I$ with $\mu_{ST}(I)\ge 0.36$. 
Furthermore, we prove a common generalization of our bounded gap result with the Green--Tao theorem.
To obtain these results, we demonstrate a Bombieri--Vinogradov type theorem for Sato--Tate primes.
\end{abstract}

\section{Introduction}\label{sec:introduction}
Let $E/\QQ$ be an elliptic curve without complex multiplication (CM), and for each prime $p$, let $a_E(p)$ denote the trace of the Frobenius endomorphism of $E/\FF_p$. 
The Sato--Tate conjecture for $\QQ$, recently proven by Barnet-Lamb, Geraghty, Harris, and Taylor \cite{BGHT11}, states that the distribution of $\cos\theta_p:= a_E(p)/(2\sqrt{p})$ is governed by the Sato--Tate measure $\mu_{ST}$. Explicitly, for $-1\le \alpha< \beta\le 1$, we have 
\[\lim_{x\to\infty}\frac{\#\{p\leq x: \alpha\leq\cos\theta_p\leq\beta\}}{\#\{p\leq x\}}=\frac{2}{\pi}\int_\alpha^\beta\sqrt{1-t^2}\,dt =: \mu_{ST}([\alpha,\beta]).\]
One can now study the distribution of primes with constraints on their trace of Frobenius.
Namely, we choose an interval $I \subseteq [-1,1]$ and consider the set of primes $p$ such that $\cos\theta_p \in I$.

There have been several important recent advances in the study of gaps between primes.
Building on the seminal work of Goldston, Pintz, and Y\i ld\i r\i m \cite{GPY09}, Zhang \cite{Z14} proved that
\[\liminf_{n\to\infty}(p_{n+1}-p_n)\leq 70\cdot 10^6,\] 
where $p_n$ denotes the $n$th prime.
Shortly thereafter, using more combinatorial methods, Maynard \cite{Ma15} synthesized the GPY framework with $k$-dimensional variant of the Selberg sieve, which allowed him to prove that $\liminf_{n\to\infty}(p_{n+m}-p_n)\ll m^3\exp(4m)$ for every $m\ge 1$.
He also strengthened Zhang's result to $\liminf_{n\to\infty}(p_{n+1}-p_n)\le 600$.
Independently, Tao (unpublished) developed the same variant of the Selberg sieve, but arrived at a slightly different conclusion. (Soon after, Polymath 8b \cite{Po14} improved this bound to $\liminf_{n\to\infty}(p_{n+1}-p_n)\le 246$.) Various authors have adapted the Maynard framework to establish bounded gaps between primes in distinguished subsets, such as primes in Beatty sequences \cite{BZ16}, and primes with a given Artin symbol \cite{T14}.

In this paper, we synthesize the Sato--Tate conjecture and the aforementioned work on gaps between primes.
Our main result is too technical to state here, so for now we state a special case for the sake of simplicity.

\begin{theorem}\label{thm:simplified-bounded-gaps}
Let $E/\QQ$ be a non-CM elliptic curve, and let $I\subseteq [-1,1]$ be a closed interval such that $\mu_{ST}(I) \ge 0.36$.
Denote by $\mc P_I$ the set of all primes $p$ satisfying $\cos\theta_p\in I$, and let $p_{I,n}$ be the $n$th prime in $\mc P_I$.
There is a constant $C_I>0$ (independent of $E$) such that for any positive integer $m$, we have that
\[
\liminf_{n \to \infty} (p_{I,n+m} - p_{I,n})\le \exp(C_I m).
\]
\end{theorem}

\begin{remark}
As \cref{thm:bounded-gaps-sato-tate} will show, there exist intervals with $\mu_{ST}(I)<0.36$ for which our result still applies.
In fact, as justified in \cref{lem:half-of-intervals}, our more general theorem holds for over $50.7\%$ of intervals $I \subseteq [-1, 1]$, if the endpoints are sampled according to the Sato--Tate distribution.
\end{remark}

\begin{example*}\label{ex:explicit-bound}
Consider $I = [-1, -5/6]$, which has Sato--Tate measure $\mu_{ST}(I)\approx 0.0398$.
For any non-CM elliptic curve $E/\QQ$, we have
\[
\liminf_{n\rightarrow \infty}(p_{I,n+1} - p_{I,n}) \leq 10^{5992}.
\]
This explicit bound is computed using \cref{thm:bounded-gaps-sato-tate} in \cref{ex:explicit-gap}.
\end{example*}
In the spirit of the Green--Tao theorem, which states that subsets of the primes with positive upper density contain arbitrarily long arithmetic progressions, we can prove a refinement of \cref{thm:simplified-bounded-gaps} by adapting methods of Pintz \cite{Pi10, Pi16, Pi17} and Vatwani and Wong \cite{VW17}.
Again, we shall only state a special case of our result; a more general version is given in \cref{thm:green-tao-sato-tate}.
To state this result, recall that a set $\mc{H}= \{h_1,\ldots,h_k\}$ of nonnegative integers is \emph{admissible} if, for any prime $p$, there exists an integer $n$ such that $p\nmid n+h_i$ for all $i=1,\ldots,k$.

\begin{theorem}\label{thm:simplified-green-tao}
If $E/\QQ$ is a non-CM elliptic curve, and $I \subseteq [-1,1]$ is a closed interval such that $\mu_{ST}(I)\ge 0.36$, then there is a constant $C_I > 0$ (independent of $E$) such that for all $m\ge 1$ the following holds. Given any admissible set $\mc{H}=\{h_1,\ldots,h_k\}$ of size $k\ge\exp(C_I m)$, there exists an $(m+1)$-element subset $\{h_1',\ldots,h_{m+1}'\}$ of $\mc{H}$ such that there are arbitrarily long arithmetic progressions in the set
\[\{n\in\NN: n+h_i'\in\mc{P}_I\emph{ for all } 1\le i\le m+1\}.\]
\end{theorem}

\begin{remark}
By a slight alteration of our argument, one can show that \cref{thm:simplified-bounded-gaps} and \cref{thm:simplified-green-tao} hold in the more general setting when the traces come from a non-CM holomorphic newform of positive even integer weight.
\end{remark}

\subsection{Overview of argument}\label{sec:overview-of-argument}
The approach of Maynard to proving bounded gaps between primes involves the careful estimation of weighted counts of primes.
In the Sato--Tate setting, there is a further constraint: we require that the normalized traces of the primes lie in $I$. 
To control the indicator function $\one_I$, one might first note that the set of Chebyshev polynomials of second kind $\{U_\ell\}_{\ell\ge 0}$ is an orthonormal basis of $L^2([-1,1], \mu_{ST})$. Thus, considering the Fourier expansion of $\one_I$ with respect to this basis, one might analyze the basis elements individually.
However, this strategy faces substantial difficulties due to a lack of understanding of the symmetric power $L$-functions $L(s,\Sym^\ell E)$ associated to $E$.

As we will see in \cref{sec:symmetric-powers}, we need to know that these $L$-functions are automorphic in order to carry through the estimates. 
Currently, automorphy has only been proven for $L(s,\Sym^\ell E)$ when $\ell\leq 8$, although it is conjectured to hold for all $\ell\in\NN$. The case of $\ell = 1$ is precisely the modularity theorem established by the combined work of Wiles \cite{W95}; Taylor and Wiles \cite{TW95}; Diamond \cite{Di96}; Conrad, Diamond, and Taylor \cite{CDT99}; and Breuil, Conrad, Diamond, and Taylor \cite{BCDT01}. The values $2\le\ell\le 8$ follow from automorphy lifting theorems for symmetric powers of cuspidal automorphic representations of $\GL_2$: $\ell = 2$ is due to Gelbart and Jacquet \cite{GJ78}, $\ell = 3$ to Kim and Shahidi \cite{KS02}, $\ell = 4$ to Kim \cite{K03}, and $\ell\in\{5,6,7,8\}$ to Clozel and Thorne \cite{CT15, CT17}. (In contrast, \emph{potential} automorphy is known for all $\ell$ \cite{BGHT11}.)

Accordingly, for an unconditional result, we can only afford to use an approximation of $\one_I$ by polynomials of degree up to $8$. We will choose this approximation to be a minorant $u_- \le \one_I$.
The idea of using such a polynomial minorant to obtain unconditional results appears in \cite{LT18}.
Because of these considerations, we establish our results unconditionally precisely for those intervals $I$ such that $\one_I$ is minorizable by a polynomial of degree at most $8$ with positive average against $\mu_{ST}$.
(\cref{app:minorization} gives a more technical discussion of minorizing indicator functions $\mathbbm{1}_{I}$ by polynomials with these constraints.)
Assuming automorphy for all $\ell$, our results hold for all intervals $I \subseteq [-1, 1]$.

We now explain the structure of our paper.
In \cref{sec:PNT-BV-sato-tate}, we first state the number-theoretic results that are required to adapt the general frameworks of Maynard and Pintz to the Sato--Tate setting; namely, we need versions of the prime number theorem and the Bombieri--Vinogradov theorem for Sato--Tate primes.
We will assume these inputs in \cref{sec:bounded-gaps} to establish a more general version of \cref{thm:simplified-bounded-gaps}, our bounded gaps result.
Similarly, in \cref{sec:green-tao} we prove a more general version of Theorem \ref{thm:simplified-green-tao}, the synthesis of our bounded-gap result with the Green--Tao theorem, assuming the results of \cref{sec:PNT-BV-sato-tate}.

All the subsequent sections are dedicated to proving the theorems in \cref{sec:PNT-BV-sato-tate}. 
Towards this, in \cref{sec:symmetric-powers} we begin our technical discussion of symmetric power $L$-functions.
\cref{sec:parameter-bounds} discusses bounds on coefficients and values of $L$-functions, estimates required for proving our analogues of the Siegel--Walfisz and Bombieri--Vinogradov theorems. \cref{sec:zero-free-siegel} establishes the Siegel--Walfisz theorem in the Sato--Tate setting, using a zero-free region and an analogue of Siegel's theorem for symmetric power $L$-functions, which follows from the work of Molteni \cite{Mo99}.
Finally, in \cref{sec:bombieri} we complete the proof of our Bombieri--Vinogradov type estimate, adapting methods of Murty and Murty \cite{MM87}.

\subsection{Conventions}
Throughout this paper we shall use the following notation. We will use the variable $p$ to index primes in sums and products. Given two functions $f,g$, we say that $f=O(g)$, or $f\ll g$, if there exists a constant $C$ such that $|f|\leq C|g|$; the subscript versions $O_a$, $\ll_a$ imply that the constant may depend on $a$. Similarly, we may use $C_a$, $c_a$ or $c'_a$ to denote constants that depend on $a$. For a positive integer $N$, its \textit{radical} is defined as $\rad(N) := \prod_{p\mid N}p$, and $P^-(N)$ denotes the smallest prime divisor of $N$. As usual, $\mu(\cdot)$ is the M\"obius function, $\varphi(\cdot)$ is the Euler totient function, and $\pi(\cdot)$ is the prime counting function. Additionally, $\one_S$ denotes the indicator function of a set $S$.

\section{Number-theoretic inputs to bounded gaps}\label{sec:PNT-BV-sato-tate}

Let $E/\QQ$ be a fixed elliptic curve without CM. For each prime $p$, $E$ has a trace of Frobenius $2\sqrt{p}\cos\theta_p$, where $\theta_p\in [0, \pi]$. 
We consider primes $p$ with $\cos\theta_p \in I$, where $I$ is a fixed, closed subinterval of $[-1, 1]$; we denote this set of primes by $\mc{P}_I$.

As discussed in \cref{sec:overview-of-argument}, we consider a minorizing polynomial $u_-(t)\le\one_I(t)$ of bounded degree, following \cite{LT18}. This will suffice for our applications on bounded gaps, as we will see in our adaptation of the Maynard framework in \cref{sec:bounded-gaps}. 
We shall expand $u_-$ with respect to the basis of Chebyshev polynomials of the second kind, which are defined as
\begin{equation} \label{eq:chebyshevdefn}
U_\ell(\cos \theta) := \frac{\sin ((\ell + 1)\theta)}{\sin \theta}.
\end{equation}
We use these polynomials due to their relationship to symmetric power $L$-functions, detailed in \cref{sec:symmetric-powers}; specifically, $U_\ell(\cos\theta_p)$ is the coefficient of $p^{-s}$ in $L(s,\Sym^\ell E)$ for $p$ not dividing the conductor of $E$.
One can easily show that $U_\ell$ is a polynomial of degree $\ell$, and that $|U_\ell(\cos \theta)| \leq \ell + 1$. Also, the Chebyshev polynomials form an orthonormal family with respect to the Sato--Tate measure; in particular, we have
\[
\int_{-1}^1 U_\ell(t) \mu_{ST}(dt) =
\begin{cases}
1, \qquad \text{$\ell$ = 0,} \\
0, \qquad \text{otherwise,}
\end{cases}
\]
so that integration against $\mu_{ST}$ picks up the coefficient of $U_0$ in the representation of a polynomial in a basis of $U_\ell$'s. Now if our minorizing polynomial $u_-$ has degree $\ell_{\max}$, we can write it as
\begin{equation}\label{eq:minorizingpolydefn}
    u_-(\cos\theta_p) = \sum_{\ell=0}^{\ell_{\max}} b_\ell U_{\ell}(\cos\theta_p),
\end{equation}
where $b_i\in\RR$.
The Sato--Tate average $b_0$ of $u_-$ will show up in the factor of the main terms in our analogues of the prime number theorem and the Bombieri--Vinogradov theorem, and ultimately in the choice of a positive parameter; therefore we require that $b_0 > 0$.
This condition is central to our argument, hence we formalize it in the following definition.
\begin{definition}\label{def:sym-minorizable}
For $\ell\geq 0$, we say\footnote{This notation is introduced in \cite{LT18}.} that a closed interval $I\subseteq [-1,1]$ is \emph{$\Sym^\ell$-minorizable} if there exists some polynomial $\sum_{j=0}^{\ell} b_jU_j$ of degree at most $\ell$ with $b_0>0$ and $b_j\in\RR$ which lower bounds $\one_I$ in the range $[-1,1]$.
\end{definition}
To extract estimates involving $U_\ell(\cos\theta_p)$ from its corresponding $L$-function for $\ell\le\ell_{\max}$, we make the following assumption, which justifies \cref{def:sym-minorizable}.

\begin{assumption}\label{assumption-symmetric-powers}
The symmetric power $L$-functions $L(s, \Sym^\ell E)$ of a non-CM elliptic curve $E/\QQ$ are automorphic for $0 \leq \ell \leq \ell_{\max}$.
\end{assumption}
\begin{remark}
Recall from \cref{sec:overview-of-argument} that \cref{assumption-symmetric-powers} has been proven for $\ell_{\max} = 8$.
\end{remark}

We shall now state the necessary number-theoretic results with weights given by the Chebyshev polynomials $U_\ell$, and then take linear combinations to obtain analogous results about $u_-$. The following analogue of the prime number for Chebyshev polynomials of the second kind is a consequence of \cite[Theorem~5.13]{IK04}, as detailed in \cref{sec:zero-free-siegel}.

\begin{theorem}\label{thm:Ul-prime-number-theorem} 
Let $E$ be a non-CM elliptic curve and $\ell\ge 1$.
Assuming that $L(s,\Sym^\ell E)$ is automorphic, we have
\[
\sum_{\substack{p \leq x}} U_\ell(\cos \theta_p) = O_{E, \ell}\Bigl(x \exp \Bigl(-c_{E, \ell} \sqrt{\log x}\Bigr)\Bigr).
\]
\end{theorem}

This implies the following prime number theorem for $u_-$.
\begin{corollary}\label{cor:prime-number-theorem-minorizing}
Under \cref{assumption-symmetric-powers}, we have that
\[
\sum_{N< p \le 2N} u_-(\cos\theta_p) = b_0 (\pi(2N) - \pi(N)) + O_{E, \ell_{\max}, u_-}\Bigl(N
\exp\Bigl( -c_{E, \ell_{\max}} \sqrt{\log N}\Bigr)\Bigr).
\]
\end{corollary}
\begin{proof}[Proof that \cref{thm:Ul-prime-number-theorem} implies \cref{cor:prime-number-theorem-minorizing}]
By \cref{eq:minorizingpolydefn} we have
\[\sum_{N < p\le 2N}u_-(\cos\theta_p)
=b_0(\pi(2N) - \pi(N))+O_{E, \ell_{\max}}\Biggl(N \sum_{\ell=1}^{\ell_{\max}}|b_\ell|\exp\Bigl(-c_{E, \ell}\sqrt{\log{ 2N}}\Bigr)\Biggr),\]
which gives the desired estimate.
\end{proof}

Next, we state an analogue of the Bombieri--Vinogradov theorem for the Chebyshev polynomials of second kind, to be proven in \cref{sec:bombieri}.
\begin{theorem}\label{conj:bombieri-sato-tate-final}
Let $E$ be a non-CM elliptic curve and $\ell \geq 1$. Assume that $L(s,\Sym^\ell E)$ is automorphic.
Then for any $0 < \theta < 1/{\max(2, \ell-1)}$, and for all $B > 0$, we have
\[\sum_{\substack{q\le x^\theta}}\sup_{\substack{(a, q) = 1\\y\le x}}\left|\sum_{\substack{p\le y\\p\equiv a\pmod{q}}}U_\ell(\cos\theta_p)\right|\ll_{B, E, \ell}x(\log x)^{-B}.\]
\end{theorem}

This allows us to deduce a similar estimate for minorizing functions $u_-$ of suitable intervals.

\begin{corollary}\label{cor:bombieri-vinogradov-minorizing}
 Under \cref{assumption-symmetric-powers}, we have that for any $0 < \theta <1/{\max(2, \ell_{\max}-1)}$ and for all $B>0$, we have
\[
\sum_{q\le x^{\theta}} \sup_{(a,q)=1} \left|\sum_{\substack{x < p \leq 2x\\ p \equiv a\pmod{q}}} u_-(\cos \theta_p) - b_0 \frac{\pi(2x) - \pi(x)}{\varphi(q)}\right|\ll_{B, E, \ell_{\max}, u_-}x(\log x)^{-B}.
\]
\end{corollary}
\begin{proof}[Proof that \cref{conj:bombieri-sato-tate-final} implies \cref{cor:bombieri-vinogradov-minorizing}]
By \cref{eq:minorizingpolydefn}, we have
\begin{align*}\sum_{\substack{x < p \le 2x\\ p \equiv a\pmod{q}}} u_-(\cos &\theta_p) - b_0 \frac{\pi(2x) - \pi(x)}{\varphi(q)} \\
&=b_0\left[\sum_{\substack{x < p \leq 2x\\ p \equiv a\pmod{q}}}1-\frac{\pi(2x) - \pi(x)}{\varphi(q)}\right]
+\sum_{\substack{x < p \le 2x\\ p \equiv a\pmod{q}}}\sum_{\ell=1}^{\ell_{\max}}b_\ell U_\ell(\cos\theta_p).
\end{align*}
Taking the supremum over $(a,q)=1$ and averaging over moduli less than $x^\theta$, we can bound the first term using the original Bombieri--Vinogradov theorem, and the second using \cref{conj:bombieri-sato-tate-final}.
\end{proof}

In the next two sections, we will use these estimates to derive our results on bounded gaps between primes from a Sato--Tate interval, and on patterns of such primes in the context of the Green--Tao theorem. Subsequently, it will only remain to prove \cref{thm:Ul-prime-number-theorem,conj:bombieri-sato-tate-final} (see \cref{sec:zero-free-siegel}, respectively \cref{sec:bombieri}).

\section{Bounded gaps for Sato--Tate}\label{sec:bounded-gaps}

In this section, we adapt the work of Maynard \cite{Ma15} in order to establish bounded gaps among the primes in $\PP_I$, for suitable intervals. Following the notation in \cref{sec:PNT-BV-sato-tate}, let $u_-$ be a minorizing polynomial of an interval $I \subseteq [-1, 1]$, with Sato--Tate average $b_0 > 0$ and degree $\ell_{\max}$. Let $E/\QQ$ be a non-CM elliptic curve with normalized traces $\cos \theta_p$, and suppose that it satisfies \cref{assumption-symmetric-powers}.

Fix an admissible set $\mc{H}= \{h_1,\ldots,h_k\}$. 
We define 
\[
\widehat{S}(N, \rho) := \sum_{N < n \leq 2N} \Biggl(\sum_{i = 1}^k \one_\PP(n + h_i) u_-(\cos \theta_{n + h_i}) - \rho\Biggr) w_n,
\]
for nonnegative weights $w_n$ to be chosen. Note that although we have not defined $\theta_{n + h_i}$ when $n + h_i$ is not prime, its value does not matter unless $n + h_i$ is prime, due to the presence of the indicator function $\one_\PP (n + h_i)$; hence we use the notation above for brevity. By our choice of $u_-$, we have:
\[
\widehat{S}(N, \rho) \leq \sum_{N < n \leq 2N} \Biggl(\sum_{i = 1}^k \one_{\PP}(n + h_i)\one_{I}(\cos\theta_{n+h_i}) - \rho\Biggr) w_n.
\]
Our goal is to show that $\widehat{S}(N, \rho) > 0$ for sufficiently large $N$.
This would imply that there are infinitely many $n \in (N, 2N]$ such that at least $\lfloor \rho + 1 \rfloor$ of the $n + h_i$ are prime and, in fact, lie in $\mc{P}_I$.

Now define $D_0 := \log \log \log N$ and $W := \prod_{p \leq D_0} p$; note that by the prime number theorem, $W \ll (\log \log N)^2$. Take $N$ large enough such that $\rad(h_i - h_j) \mid W$ for all $1 \leq i, j \leq k$, $i \neq j$. By admissibility of $\mc{H}$, we can choose $v_0\pmod{W}$ such that $(v_0+h_m,W)=1$ for all $m$. We define our weights $w_n$ by
\[
w_n := 
\begin{cases}
\bigl(\sum_{d_i \mid n + h_i \forall i} \lambda_{d_1, \ldots, d_k}\bigr)^2, \qquad & \text{if } n \equiv v_0 \pmod{W}, \\
0, \qquad &\text{otherwise},
\end{cases}
\]
where we choose $\lambda_{d_1,\dots,d_k}$ in \cref{prop:maynard-integral-asymptotics}.
Then we let
\begin{align}\label{eq:S1S2defn}
S_1 := \sum_{\substack{N < n\le 2N \\ n \equiv v_0 \pmod{W}}} w_n,\qquad\quad \widehat{S_2} := \sum_{\substack{N < n\le 2N \\ n \equiv v_0 \pmod{W}}} \Biggl(\sum_{i = 1}^k \one_{\PP}(n + h_i) u_-(\cos\theta_{n+h_i})\Biggr) w_n,
\end{align}
so that $\widehat{S}(N, \rho) = \widehat{S_2} - \rho S_1$. 
Towards showing that this difference is positive for large $N$, we have the following estimates, analogous to those in \cite[Proposition~4.1]{Ma15}.
\begin{proposition}\label{prop:maynard-integral-asymptotics}
Suppose that \cref{assumption-symmetric-powers} holds true, and let $0 < \theta < \wt{\theta}:=1/{\max(2, \ell_{\max} - 1)}$. 
Let $R := N^{\theta/2 - \delta}$ for some small fixed $\delta > 0$. 
Let $F : [0, 1]^k \to \RR$ be a smooth function supported on $\{(x_1, \ldots, x_k) \in [0, 1]^k : \sum_{i=1}^k x_i \leq 1\}$, and define
\[
\lambda_{d_1, \ldots, d_k} := \Biggl(\prod_{i=1}^k \mu(d_i)d_i \Biggr) \sum_{\substack{r_1, \ldots, r_k \\ d_i \mid r_i \\ (r_i, W) = 1}} 
\frac{\mu\bigl(\prod_{i=1}^k r_i\bigr)^2}{\prod_{i=1}^k \varphi(r_i)} F\biggl(\frac{\log r_1}{\log R}, \ldots, \frac{\log r_k}{\log R} \biggr).
\]
In particular, $\lambda_{d_1, \ldots, d_k} = 0$ unless $\prod_{i = 1}^k d_i$ is at most $R$, coprime with $W$, and squarefree. Then we have that
\begin{align*}
S_1 &= \frac{(1+o(1))\varphi(W)^{k}N(\log R)^{k}}{W^{k+1}} I_{k}(F), \\
\widehat{S_2} &= b_0 \frac{(1 + o(1))\varphi(W)^k N (\log R)^{k+1}}{W^{k+1} \log N} \sum_{m=1}^k J_k^{(m)}(F),
\end{align*}
where $I_k(F)$ and $J_k^{(m)}(F)$ are iterated integrals defined in \cite[Proposition~4.1]{Ma15}, and we assume that they are positive.
\end{proposition}

\begin{remark}
Since $S_1$ coincides with Maynard's notation \cite{Ma15}, only the estimate for our $\widehat{S_2}$ is new.
\end{remark}

\begin{remark}
The errors $o(1)$ in the asymptotics of $S_1$ and $\widehat{S_2}$ above are in fact $O(1/D_0)$ with an absolute implied constant, and in this section we take $D_0 = \log \log \log N$. However, the terms $O(1/D_0)$ are valid even when $D_0$ is a large enough constant depending only on $k$ and $\mc H$; we will need this different choice in \cref{sec:green-tao}.
\end{remark}

\begin{proof}
In light of \cite[Proposition~4.1]{Ma15}, it suffices to consider $\widehat{S_2}$. The implied constants in what follows will depend on $\mc{H}$ and $k$. Also, we will often write $\mb{d}$ and $\mb{e}$ instead of $(d_1, \ldots, d_k)$ and $(e_1, \ldots, e_k)$ respectively, for brevity. Let us decompose $\widehat{S_2} = \sum_{m = 1}^k \widehat{S_2}^{(m)}$, where we define
\begin{align*}
\widehat{S_2}^{(m)} :=& \sum_{\substack{N < n \leq 2N \\ n \equiv v_0 \pmod{W}}} \one_{\PP}(n + h_m) u_-(\cos\theta_{n + h_m}) \left(\sum_{d_i \mid n + h_i \forall i} \lambda_{\mb d}\right)^2 \\ 
=& \sum_{\mb{d}, \mb{e}} \lambda_{\mb{d}} \lambda_{\mb{e}} \sum_{\substack{N < n \leq 2N \\ n \equiv v_0 \pmod{W} \\ [d_i, e_i] \mid n + h_i}} \one_{\PP}(n + h_m) u_-(\cos\theta_{n+h_m}).
\end{align*}
As in the proof of \cite[Proposition~4.1]{Ma15}, the sum restricts to the case when $W, [d_i, e_i], [d_j, e_j]$ are all coprime. In that case, by the Chinese remainder theorem, the sum can be rewritten with $n + h_m$ lying in a single residue class $a$ mod $q := W \prod_{i = 1}^k [d_i, e_i]$. Moreover, the inner sum (weighted by $\lambda_{\mb{d}} \lambda_{\mb{e}}$) is seen to vanish unless $d_m = e_m = 1$, and in that case $a$ must be coprime with $q$. For such $\mb{d}$ and $\mb{e}$, the inner sum becomes
\[
\sum_{\substack{N < n \le 2N \\ n + h_m \equiv a \pmod{q}}} \one_{\PP}(n + h_m) u_-(\cos\theta_{n + h_m}).
\]
Note that one has
\[
\left(\sum_{\substack{N < n \le 2N \\ n + h_m \equiv a \pmod{q}}} \one_{\PP}(n + h_m) u_-(\cos\theta_{n + h_m})\right)-\left(\sum_{\substack{N < n \le 2N \\ n  \equiv a \pmod{q}}} \one_{\PP}(n) u_-(\cos\theta_{n})\right) = O(1),
\]
since $h_m$ is finite and $u_-$ is bounded (it is a polynomial on a compact interval). 
We denote
\[
E(N,q) := 1+\sup_{(a,q)=1}\left|\sum_{\substack{N < n \le 2N \\ n\equiv a \pmod{q}}} \one_{\PP}(n) u_-(\cos\theta_{n}) - b_0 \frac{\pi(2N) - \pi(N)}{\varphi(q)}\right|.
\]
Putting these together, we have
\[
\sum_{\substack{N < n \le 2N \\ n + h_m \equiv a \pmod{q}}} \one_{\PP}(n + h_m) u_-(\cos\theta_{n + h_m})= b_0 \frac{\pi(2N) - \pi(N)}{\varphi(q)} + O(E(N,q)).
\]
Plugging this into $\widehat{S_2}^{(m)}$ and using the multiplicativity of $\varphi$, this implies
\[
\widehat{S_2}^{(m)} = b_0 \frac{\pi(2N) - \pi(N)}{\varphi(W)} {\sideset{}{'}\sum_{\substack{\mb{d}, \mb{e} \\ d_m =1\\ e_m = 1}}} \frac{\lambda_\mb{d} \lambda_\mb{e}}{\prod_{i = 1}^k \varphi([d_i, e_i])} + O\Biggl({\sum_{\mb{d}, \mb{e}}} '|\lambda_\mb{d} \lambda_\mb{e}| E(N, q) \Biggr),
\]
where ${\sum}'$ denotes the restriction that all $W, [d_i, e_i], [d_j, e_j]$ are pairwise coprime.

We first bound the error term using \cref{cor:bombieri-vinogradov-minorizing}. 
As in \cite[(5.9)]{Ma15} we have $\lambda_{\max} \ll y_{\max} (\log R)^{k}$, where $\lambda_{\max} := \sup_\mb{d} \lambda_\mb{d}$ and $y_{\max}$ is defined as in \cite[Lemma 5.1]{Ma15}. This gives a bound for $|\lambda_\mb{d} \lambda_\mb{e}|$, and we can restrict the range of possible $q$'s to squarefree $q = r < R^2W$. But for any given squarefree $r$, there are at most $\tau_{3k}(r)$ choices of $d_1, \ldots , d_k, e_1, \ldots, e_k$ such that $r= W\prod_{i=1}^{k} [d_i, e_i]$, where $\tau_{3k}(r)$ is the number of ways to write $r$ as a product of $3k$ positive integers. Thus we get an error term of
\[
\ll y_{\max}^2(\log R)^{2k} \sum_{r<R^2 W} \mu(r)^2 \tau_{3k}(r) E(N,r).
\]

By Cauchy--Schwarz, as well as the trivial bound $E(N,q)\ll_{\ell_{\max}, u_-} N/\varphi(q)$, we get an error of
\begin{align*}
\ll y_{\max}^2 (\log R)^{2k}\left(\sum_{r<R^2 W}\mu(r)^2 \tau_{3k}^{2}(r) \frac{N}{\varphi(r)}\right)^{1/2}\left(\sum_{r<R^2 W} \mu(r)^2 E(N,r)\right)^{1/2}.
\end{align*}
Note that $R^2 W\ll N^{\theta}$. The middle factor is bounded by $N^{1/2}(\log N)^{3k/2}$. By \cref{cor:bombieri-vinogradov-minorizing}, we have that the last factor is bounded by $(N(\log N)^{-A})^{1/2}$ for all $A>0$. So the total error is $\ll_A y_{\max}^2N(\log N)^{-A}$, concluding our analysis of the error term. But the main term in our expression of $\widehat{S_2}^{(m)}$ is exactly the same as \cite[(5.18)]{Ma15}. Hence, our asymptotic for $\widehat{S_2}$ is precisely
\[
\widehat{S_2} = b_0 \frac{(1 + o(1))\varphi(W)^k N (\log R)^{k+1}}{W^{k+1} \log N} \sum_{m=1}^k J_k^{(m)}(F).
\]
This finishes the proof.
\end{proof}

Next, we obtain a result analogous to  \cite[Proposition~4.2]{Ma15}, which is the final technical result before proving bounded gaps in Sato--Tate intervals.

\begin{proposition}\label{prop:maynard-primes-in-interval}
Suppose that an interval $I \subseteq [-1, 1]$ is $\Sym^{\ell_{\max}}$-minorizable for some $\ell_{\max} \geq 0$ by a polynomial with average $b_0$ against $\mu_{ST}$. Assume the hypothesis and notation of \cref{prop:maynard-integral-asymptotics}, and let $\{h_1, \ldots, h_k\}$ be an admissible set. Denote by $\mc{S}_k$ the set of all Riemann-integrable real functions supported on $\{(x_1, \ldots, x_k) \in [0, 1]^k : \sum_{i = 1}^k x_i \leq 1\}$. Define
\[
M_k := \sup_{F \in \mc{S}_k} \frac{\sum_{m=1}^k J_k^{(m)}(F)}{I_k(F)}, \qquad \qquad r_k := \biggl\lceil b_0 M_k \frac{\theta}{2} \biggr\rceil.
\]
Then there are infinitely many integers $n$ such that at least $r_k$ of the numbers $n + h_i$ lie in $\mc{P}_I$.
\end{proposition}
\begin{proof}
This is the same argument as in \cite[Proposition~4.2]{Ma15} \emph{mutatis mutandis}, where one must account for the extra factor of $b_0$ coming from the estimate of $\widehat{S_2}$.
\end{proof}

Using the last two propositions, we are ready to prove \cref{thm:bounded-gaps-sato-tate}, our main result on bounded gaps in Sato--Tate intervals.

\begin{theorem}\label{thm:bounded-gaps-sato-tate}
Suppose that an interval $I \subseteq [-1, 1]$ is $\Sym^{\ell_{\max}}$-minorizable, for some $\ell_{\max}\geq 0$, by a polynomial with average $b_0$ against $\mu_{ST}$. Let $E$ be a non-CM elliptic curve, and assume that $L(s,\Sym^\ell E)$ is automorphic for all $0\le\ell\le\ell_{\max}$.
Define $\widetilde\theta:=1/\max(2,\ell_{\max}-1)$.
Then for every positive integer $m$, we have
\[\liminf_{n\to\infty}(p_{I,n+m}-p_{I,n})\ll\frac{m}{b_0\widetilde\theta}\exp\biggl(\frac{2m}{b_0\widetilde\theta}\biggr),\]
where the implied constant is absolute.
\end{theorem}

\begin{proof}
By \cref{prop:maynard-primes-in-interval}, it suffices to find an admissible tuple $\{h_1, h_2, \dots, h_k\}$ and choose some $\theta<\wt{\theta}$ such that 
\[
\max_{1\le i, j \le k} |h_i - h_j|\ll \frac{m}{b_0 \theta}\exp\biggl(\frac{2m}{b_0\theta}\biggr),
\]
for some $k$ with $m+1\le \lceil b_0M_k\theta/2\rceil$.
In particular, it suffices to choose $k$ and some $\theta<\wt{\theta}$, depending on $k$, such that $m< b_0 M_k \theta/2$.
On the other hand, by \cite[Theorem~23]{Po14} there is an absolute, effective constant $C > 0$ such that $M_k \ge \log k - C$ for all $k \ge C$.
In light of this, we choose
\[
k = \max\biggl(3,C, \exp\biggl(\frac{2m}{b_0\wt{\theta}}+(C+1)\biggr)\biggr),\qquad\quad \theta = \biggl(1-\frac{1}{\log k}\biggr)\wt{\theta}.
\]
We can take $\max_{1\le i, j \le k} |h_i - h_j| \ll k \log k$ by letting $\mc H$ contain the first $k$ primes greater than $k$. It follows that
\begin{align*}
    b_0 M_k \frac{\theta}{2}\ge b_0 (\log k - C)\frac{\theta}{2} = b_0(\log k - C)\biggl(1 - \frac{1}{\log k}\biggr) \frac{\wt{\theta}}{2} > b_0 \frac{\wt{\theta}}{2}(\log k - C - 1) \ge m,
\end{align*}
hence we have bounded gaps $p_{I,n+m} - p_{I,n}$ of size 
\[
\ll k \log k \ll  \frac{m}{b_0 \wt{\theta}} \exp \biggl(\frac{2m}{b_0\wt{\theta}}\biggr)
\]
for all $m$.
\end{proof}

\begin{proof}[Proof of \cref{thm:simplified-bounded-gaps}]
Choose $\ell_{\max} = 8$, so that $L(s, \Sym^\ell E)$ is automorphic for all $\ell \le \ell_{\max}$ and for all non-CM elliptic curves $E$ \cite{CT17}. By the computations in \cref{lem:minorization-big-measure}, all intervals $I \subseteq [-1, 1]$ with $\mu_{ST}(I) \ge 0.36$ can be minorized as in the hypothesis of \cref{thm:bounded-gaps-sato-tate}. Hence \cref{thm:simplified-bounded-gaps} follows for large enough $C_I$ (note that $\widetilde\theta = 1/7$ and $b_0$ only depends on $I$).
\end{proof}

\section{Green--Tao for patterns of Sato--Tate primes}\label{sec:green-tao}

Our goal in this section is to prove \cref{thm:simplified-green-tao}, an analogue of the Green--Tao theorem in the setting of Sato--Tate primes in bounded gaps, assuming the results in \cref{sec:PNT-BV-sato-tate}. In fact, we will prove a more general version of this result, stated in \cref{thm:green-tao-sato-tate}.
Our proof relies on methods developed by Pintz \cite{Pi10} and Vatwani and Wong \cite{VW17}.

We recall our setup.
We fix $E$, a non-CM elliptic curve over $\QQ$, and $I$, a closed interval in $[-1,1]$ which is minorized by the polynomial \cref{eq:minorizingpolydefn} with $b_0>0$, where $\ell_{\max}$ satisfies \cref{assumption-symmetric-powers}.  
We choose $0 < \theta < 1/\max(2,\ell_{\max}-1)$ as the ``level of distribution of $\mc{P}_I$'', and $\delta>0$ is a small constant. We fix $\mc{H} = \{h_1,\ldots,h_k\}$ an admissible set. 
For a positive integer $N$, let $R = N^{{\theta}/{2}-\delta}$. 
Let $\lambda_{\mb{d}}$, and $w_n$ be defined as in \cref{prop:maynard-integral-asymptotics}, taking $F$ to satisfy the same hypotheses. However, $D_0$ will be chosen as a sufficiently large constant depending only on $\mc{H}$, rather than $\log\log\log N$.

We first state a theorem of Pintz which characterizes a sufficient condition for the existence of arithmetic progressions of arbitrary length.
\begin{theorem}[{\cite[Theorem~5]{Pi10}}]\label{thm:Pintz}
Let $\mc H = \{h_1, \ldots, h_k\}$ be an admissible set, and suppose that a set $S(\mc H)$ of positive integers satisfies the two conditions
\begin{align*}
    P^-\Biggl( \prod_{i = 1}^k (n + h_i) \Biggr) \ge n^{c_1(k)}\quad\emph{for all }n \in S(\mc H),
    \qquad
     \#\{n \leq x : n \in S(\mc H)\} \ge c_2(k)\frac{x}{(\log x)^k}\emph{ for all }x,
\end{align*}
for some constants $c_1(k), c_2(k) > 0$. Then for every positive integer $t$, $S(\mc H)$ contains infinitely many $t$-term arithmetic progressions.
\end{theorem}

In light of \cref{thm:Pintz}, the proof of \cref{thm:green-tao-sato-tate} now reduces to finding a suitable subset $\mc{H'}\subseteq \mc{H}$ of size $m+1$ as well as $S(\mc{H'})$ that satisfies all conditions of this theorem. 
Our approach follows that of Vatwani and Wong \cite{VW17}, replacing Chebotarev sets of primes with Sato--Tate prime sets $\mc{P}_I$.
Recalling the definitions of $S_1$ and $\widehat{S_2}$ in \cref{eq:S1S2defn}, we define, for a constant $c_1(k)$, the following sub-sums:
\begin{align*}
S_1^-(N,\mc P_I)&:=\sum_{\substack{N< n\le 2N\\n\equiv v_0\pmod W\\P^-(\prod_{i=1}^k(n+h_i))<n^{c_1(k)}}}w_n,\\
\widehat{S_2}^-(N,\mc P_I)&:=
\sum_{\substack{N < n\le 2N\\n\equiv v_0\pmod{W}\\P^-(\prod_{i=1}^k(n+h_i))<n^{c_1(k)}}}\Biggl(\sum_{i = 1}^k \one_\PP(n + h_i) u_-(\cos \theta_{n + h_i})\Biggr)w_n.
\end{align*}
In other words, $S_i^-(N,\mc P_I)$ is the contribution to $S_i$ coming from those terms such that some $n+h_j$ has a prime factor less than $n^{c_1(k)}$. 

We now provide a roadmap for the remainder of this section.
Following this paragraph we state \cref{lem:bound-on-S1pj}, which will help bound $S_1^-(N,\mc P_I)$ and $\widehat{S_2}^-(N,\mc P_I)$. 
These estimates will imply that the main contribution to the sums $S_i$ come from those terms such that $n+h_j$ has large prime factors for all $j$. 
In this way, we will be able to bound from below the quantity of $n$ with $P^-(\prod_{j=1}^k (n+h_j)) > n^{c_1(k)}$. 
Subsequently, \cref{prop:average-Pintz-condition,prop:Pintz-condition} will lead to a suitable choice of $\mc{H'}\subseteq \mc{H}$ and $S(\mc{H})$ which satisfy the hypotheses of \cref{thm:Pintz}.
Finally, we will conclude this section with a proof of \cref{thm:green-tao-sato-tate}.

\begin{lemma}\label{lem:bound-on-S1pj}
Given any $D_0 > \sup\{|h_i-h_j|:1\leq i,j\leq k\}$, define $W := \prod_{p\leq D_0}p$. For any $1\leq j\leq k$ and any prime $p < R$, we have for sufficiently large $N$ (in terms of $k, \mc{H}, D_0$) that
\[
S_{1,p}^{(j)} := \sum_{\substack{N < n\le 2N \\ n\equiv v_0 \pmod{W}\\ p\mid n+h_j}} w_n\ll_{F,k} \biggl(\frac{(\log p)^2}{p(\log R)^2}+\frac{1}{p^3}\biggr) \frac{N(\log R)^k}{W}.
\]
\end{lemma}
\begin{remark}
The function $F$ here, which satisfies the hypotheses of \cref{prop:maynard-integral-asymptotics}, must in fact be smooth on all of $\RR^k$.
Note we can later choose $F$ to be sufficiently close to the choice in Maynard \cite{Ma15} and Polymath 8b \cite{Po14}, since these smooth functions well-approximate Riemann-integrable functions supported on the simplex.
\end{remark}
\begin{proof}
Note that if $p\le D_0$ then the sum is in fact empty by definition of $v_0$. We now assume $p > D_0$. Without loss of generality, let $j = 1$. By definition, we have
\begin{align*}
S_{1,p}^{(1)} &= \sum_{\substack{N<n\le 2N\\n\equiv v_0\pmod{W}\\p|n+h_1}}\left(\sum_{d_i|n+h_i}\lambda_{\mb{d}}\right)^2
= \sum_{\mb{d}, \mb{e}}\lambda_{\mb{d}}\lambda_{\mb{e}}\sum_{\substack{N<n\le 2N\\n\equiv v_0\pmod{W}\\ [d_j,e_j]|n+h_j\forall j\\p|n+h_1}}1\\
&= \frac{N}{pW}{\sum_{\mb{d}, \mb{e}}}'\frac{\lambda_{\mb{d}}\lambda_{\mb{e}}}{\frac{[d_1,e_1,p]}{p}\prod_{j=2}^k[d_j,e_j]}+O\Bigl(\lambda_{\max}^2R^2(\log R)^{2k}\Bigr),
\end{align*}
where $\sum'$ indicates that we are restricting the sum such that $p\nmid d_j, e_j$ for $2\le j\le k$ and $\gcd(d_i, e_j) = 1$ for all $i\not= j$. 
We do this because it can be checked that these are precisely the conditions (beyond the support restrictions on $\lambda_{\mb{d}}$) necessary to have the inner sum furthest to the right of the top line be nonempty and resolve into a single residue class. We are implicitly using the lower bound on $D_0$.

Since $R = N^{\theta/2-\delta}$ and $\theta < 1/2$, while $\lambda_{\max}\ll_{F,k} (\log N)^k$ as in Maynard \cite{Ma15}, this error term is negligible. Let the sum in the main term be denoted $T_{1, p}^{(1)}$.
Now define $g(n)$, supported only on squarefree integers, such that $g(n) = [n, p]/p$ on those values. 
We can easily check that $g$ is multiplicative, and $g(\ell) = 1 + (\ell - 1)\one_{\ell \neq p}$ for prime $\ell$. 
Hence define $\widetilde{g}(\ell) := (\ell - 1)\one_{\ell \neq p}$ for prime $\ell$, extending via multiplicativity and letting it be zero on non-squarefree integers. We see that $\widetilde{g}(n)\le\varphi(n)$ always, and that if $n$ is squarefree, then $g(n) = \sum_{d|n} \widetilde{g}(d)$.
Thus, since the support is restricted to squarefree $d_i$ (with bounded product $\prod_{i=1}^k d_i\leq R$ and relatively prime to $W$), we see
\begin{align*}
T_{1, p}^{(1)} &= {\sum_{\mb{d}, \mb{e}}}'\frac{\lambda_{\mb{d}}\lambda_{\mb{e}}}{g([d_1,e_1])\prod_{j=2}^k[d_j,e_j]}\\
&= {\sum_{\mb{d}, \mb{e}}}'\frac{\lambda_{\mb{d}}\lambda_{\mb{e}}}{g(d_1)g(e_1)\prod_{j=2}^k(d_je_j)}\sum_{u_i|d_i, e_i}\widetilde{g}(u_1)\prod_{j = 2}^k\varphi(u_j)\\
&= \sum_{\mb{u}}\widetilde{g}(u_1)\prod_{j=2}^k\varphi(u_j){\sideset{}{'}\sum_{u_i|d_i, e_i}}\frac{\lambda_{\mb{d}}\lambda_{\mb{e}}}{g(d_1)g(e_1)\prod_{j=2}^k(d_je_j)}\\
&= \sum_{\mb{u}}\widetilde{g}(u_1)\prod_{j=2}^k\varphi(u_j)\sum_{s_{1,2},\ldots,s_{k,k-1}}\prod_{\substack{1\le i, j\le k\\i\neq j}}\mu(s_{i,j})
\sideset{}{^\dag}\sum_{\substack{\mb{d}, \mb{e}\\u_i|d_i,e_i\forall i\\s_{i,j}|d_i,e_j\forall i\neq j}}
\frac{\lambda_{\mb{d}}\lambda_{\mb{e}}}{g(d_1)g(e_1)\prod_{j=2}^k(d_je_j)},
\end{align*}
using the same technique as in Maynard \cite{Ma15}. Here $\sum^\dag$ means that only the condition $p\nmid d_j, e_j$ for $2\le j\le k$ is maintained.

We now restrict the $s_{i, j}$ to be coprime to $u_i$ and $u_j$, since terms with $s_{i,j}$ not coprime to $u_i$ or $u_j$ give zero by the support restrictions. Similarly, we can restrict to $s_{i, j}$ coprime to $s_{i, a}, s_{b, j}$ when $a\neq j$ and $b\neq i$. Furthermore, note that $p\nmid d_j$ for $2\le j\le k$ means that all $p\nmid s_{i,j}$, so we restrict in this way as well. Denote summation over $s_{i,j}$ with these restrictions by $\sum^\ast$. Furthermore, since $\widetilde{g}(u_1) = 0$ when $p|u_1$, and since $p\nmid u_j$ for $2\le j\le k$ by the other restrictions above, we can restrict the summation over $u$ so that $p\nmid u_i$ for all $1\le i\le k$.
This allows us to compute $T_{1,p}^{(1)}$ as follows.
\begin{align}\label{eq:T1p1_expanded}
T_{1,p}^{(1)} &= \sum_{\substack{\mb{u}\\ p\nmid u_i}} \widetilde{g}(u_1)\prod_{j=2}^k\varphi(u_j){\sideset{}{^\ast}\sum_{s_{1,2},\ldots,s_{k,k-1}}}\prod_{\substack{1\le i, j\le k\\i\neq j}}\mu(s_{i,j}){\sideset{}{^\dag}\sum_{\substack{\mb{d}, \mb{e}\\u_i|d_i,e_i\forall i\\s_{i,j}|d_i,e_j\forall i\neq j}}}\frac{\lambda_{\mb{d}}\lambda_{\mb{e}}}{g(d_1)g(e_1)\prod_{j=2}^k(d_je_j)},\nonumber\\
&= \sum_{\substack{\mb{u}\\ p\nmid u_i}}\widetilde{g}(u_1)\prod_{j=2}^k\varphi(u_j){\sideset{}{^\ast}\sum_{s_{1,2},\ldots,s_{k,k-1}}}\left(\prod_{\substack{1\le i, j\le k\\i\neq j}}\mu(s_{i,j})\right)\frac{\mu(a_1)\mu(b_1)}{\widetilde{g}(a_1)\widetilde{g}(b_1)}\prod_{j=2}^k\frac{\mu(a_j)\mu(b_j)}{\varphi(a_j)\varphi(b_j)}z_{\mb{a}}z_{\mb{b}},\nonumber\\
&= \sum_{\substack{\mb{u}\\ p\nmid u_i}}\prod_{i=1}^k\varphi(u_i){\sideset{}{^\ast}\sum_{s_{1,2},\ldots,s_{k,k-1}}}\left(\prod_{\substack{1\le i,j\le k\\i\neq j}}\mu(s_{i,j})\right)\prod_{j=1}^k\frac{\mu(a_i)\mu(b_i)}{\varphi(a_i)\varphi(b_i)}z_{\mb{a}}z_{\mb{b}},
\end{align}
where $a_i = u_i\prod_{j\neq i}s_{i,j}$, $b_i = u_i\prod_{j\neq i}s_{j,i}$, and where we define the quantity
\[z_{\mb{r}} := \mu(r_1)\widetilde{g}(r_1)\Bigg(\prod_{j=2}^k\mu(r_j)\varphi(r_j)\Bigg)\sum_{\substack{r_i|d_i\\p\nmid d_j, \forall j\ge 2}}\frac{\lambda_{\mb{d}}}{g(d_1)\prod_{j = 2}^k d_j},\]
which is supported only on $\mb{r}$ such that $r_1\cdots r_k$ is less than $R$, squarefree, and coprime to $pW$.

Note that division by factors of $\widetilde{g}$ (which are potentially zero) is not invalid here, as we have restricted summation appropriately. Additionally, we used that $\widetilde{g}$ and $\varphi$ agree on squarefree numbers relatively prime to $p$. This also implies that if $p\nmid r_1$, in $z_{\mb{r}}$ we can turn the $\widetilde{g}$ into a $\varphi$ without consequence.

Now we bound the remaining sum \cref{eq:T1p1_expanded} by
\begin{align*}
|T_{1,p}^{(1)}|\leq z_{\max}^2\left(\sum_{\substack{u<R\\(u,pW)=1}}\frac{\mu(u)^2}{\varphi(u)}\right)^k\left(\sum_{\substack{s\ge 1\\p\nmid s}}\frac{\mu(s)^2}{\varphi(s)^2}\right)^{k^2-k}
\ll_k z_{\max}^2(\log R)^k,
\end{align*}
where $z_{\max} = \sup_{\mb{r}}|z_{\mb{r}}|$. Now we compute $z$ in terms of $r$, recalling that we chose $\lambda$ in the proof of \cref{prop:maynard-integral-asymptotics} by choosing
\[y_{\mb{r}} = \Biggl(\prod_{i=1}^k\mu(r_i)\varphi(r_i)\Biggr)\sum_{\mb{d}:\,r_i|d_i}\frac{\lambda_{\mb{d}}}{\prod_{i=1}^kd_i}\]
to be
\[y_{r_1, \ldots, r_k} = F\biggl(\frac{\log r_1}{\log R},\ldots,\frac{\log r_k}{\log R}\biggr)\]
for $\mb{r}$ such that $r_1\cdots r_k$ is at most $R$, squarefree, and relatively prime to $W$; $y_{\mb{r}} = 0$ otherwise. Here $F$ is our chosen smooth function. For $1\le j\le k$ define 
\begin{align*}
\alpha_{\mb{r},j} = \Biggl(\prod_{i=1}^k\mu(r_i)\varphi(r_i)\Biggr)\sum_{\substack{r_i|d_i\\p\nmid d_i}}\frac{\lambda_{p^{\{j\}}\odot\mb{d}}}{\prod_{i=1}^kd_i}, \qquad \alpha_{\mb{r},0} = \Biggl(\prod_{i=1}^k\mu(r_i)\varphi(r_i)\Biggr)\sum_{\substack{r_i|d_i\\p\nmid d_i}}\frac{\lambda_{\mb{d}}}{\prod_{i=1}^kd_i},
\end{align*}
where $p^{\{j\}}\odot \mb d$ is the vector of length $k$ with coordinates $(d_1,\ldots,d_{j-1},pd_j,d_{j+1},\ldots,d_k)$. We compute for $1\le j\le k$ that
\[\alpha_{\mb{r},j} = -\frac{p}{p-1}y_{p^{\{j\}}\odot\mb{r}}, \qquad \alpha_{\mb{r},0} = y_{\mb{r}} + \sum_{j=1}^k \frac{y_{p^{\{j\}}\odot\mb{r}}}{p-1}.\]
Noting that $g$ is the identity for squarefree numbers not divisible by $p$, we now have for $p\nmid r_1\cdots r_k$ that
\begin{align*}
z_{\mb{r}} 
= \Biggl(\prod_{i=1}^k\mu(r_i)\varphi(r_i)\Biggr)\left(\sum_{\substack{r_i|d_i\\p\nmid d_i}}\frac{\lambda_{\mb{d}}}{\prod_{i=1}^kd_i} + \sum_{\substack{r_i|d_i\\p\nmid d_i}}\frac{\lambda_{pd_1,d_2,\ldots,d_k}}{\prod_{i=1}^kd_i}\right)
= \alpha_{\mb{r},0} + \alpha_{\mb{r},1} = y_{\mb r} - y_{p^{\{1\}}\odot \mb r} + \sum_{j=2}^k \frac{y_{p^{\{j\}}\odot \mb r}}{p-1}.
\end{align*}

We now return to the bounding. Since $F$ is smooth and compactly supported, by the mean value theorem we obtain
\begin{align*}
    z_{\mb{r}}
    \le \biggl(\frac{\log p}{\log R}\biggr)\Biggl(\sup_{[0,1]^k} |\nabla F|\Biggr) + \frac{k-1}{p-1}\Biggl(\sup_{[0,1]^k}F\Biggr)\ll_{F,k}\frac{\log p}{\log R} + \frac{1}{p}.
\end{align*}
\[\]
Hence, our sum is bounded in absolute value as
\begin{align*}
T_{1,p}^{(1)}\ll_{F,k}z_{\max}^2(\log R)^k\ll_{F,k}\biggl(\frac{(\log p)^2}{(\log R)^2}+\frac{1}{p^2}\biggr)(\log R)^k.
\end{align*}
Thus our original expression is bounded as
\[S_{1,p}^{(1)}\ll_{F,k}\biggl(\frac{(\log p)^2}{p(\log R)^2} + \frac{1}{p^3}\biggr)\frac{N(\log R)^k}{W},\]
as desired.
\end{proof}

The next proposition estimates the part of the difference $\widehat{S_2}-\rho S_1$ for which $\prod_{i=1}^k(n+h_i)$ has large prime factors, which we will see comprises the dominant behavior.

\begin{proposition}\label{prop:average-Pintz-condition}
Assume the notation and hypotheses from \cref{lem:bound-on-S1pj}. Let $m$ be a positive integer such that $k\ge C\exp((2m)/(b_0\theta))$, where the constant $C$ is absolute. Then there exists a choice of $D_0$ only depending on $\mc{H}$ and $k$ and a choice of $c_1(k)$ sufficiently small such that
\[
\widehat{S_2}^+ (N, \mc P_I) - m S_1^+(N, \mc P_I) \gg_k N(\log R)^k,
\]
where $S_1^+ := S_1 - S_1^-$ and $\widehat{S_2}^+ = \widehat{S_2} - \widehat{S_2}^-$.
\end{proposition}
\begin{proof}
This follows from \cref{lem:bound-on-S1pj} and the remark following \cref{prop:maynard-integral-asymptotics}. We use the same arguments as in {\cite[Lemmas~5.2,~5.3,~and~5.4]{VW17}}, the key points being that
\[\sum_{p<n^{c_1(k)}}\frac{(\log p)^2}{p(\log R)^2}\ll\frac{(c_1(k)\log n)^2}{(\log R)^2},\qquad\sum_{D_0 < p < n^{c_1(k)}}\frac{1}{p^3}\ll\frac{1}{D_0^2},\qquad\frac{W}{\phi(W)} = O(\log D_0),\]
and that $S_{1,p}^{(j)} = 0$ for all $j$ and $p\le D_0$.
Working out the constants, this means that we can choose an appropriately large $D_0$, depending only on $k$ and $\mc H$, and then can take $c_1(k)$ small enough to ensure that the contribution from $S_1^-$, and hence also $\widehat{S_2}^-$, is small.
\end{proof}

\begin{proposition}\label{prop:Pintz-condition}
Given $k$ and $\mc{H}$, there exists $c_1(k)$ chosen sufficiently small so that the following holds: if we define
\[S_{\mc{P}_I}(\mc{H}) := \biggl\{n\in\NN:\emph{ at least }m+1\text{ of }n+h_i\emph{ are in }\mc{P}_I, P^-\Biggl(\prod_{i=1}^k(n+h_i)\Biggr)\ge n^{c_1(k)}\biggr\},\]
then we have
\[\#\{n\le x: n\in S_{\mc{P}_I}(\mc{H})\}\ge c_{\mc P_I}(k,\mc H)x(\log x)^{-k}\]
for some constant $c_{\mc{P}_I}(k, \mc H)$.
\end{proposition}
\begin{proof}
Note first that
\[
\sum_{i = 1}^k \one_\PP (n + h_i) u_-(\cos \theta_{n + h_i}) - m \leq k - m \leq k
\]
for all positive integers $n$. Define \[
S'_{\mc{P}_I}(\mc{H}) = \biggl\{n\in\NN: \sum_{i = 1}^k \one_\PP (n + h_i) u_-(\cos \theta_{n + h_i}) - m > 0 \text{ and } P^-\Biggl(\prod_{i=1}^k (n+h_i)\Biggr) \ge n^{c_1(k)}\biggr\}.\]
By the definition of $u_-$, the condition $\sum_{i = 1}^k \one_\PP (n + h_i) u_-(\cos \theta_{n + h_i}) - m >0$ implies that for each $n\in S_{\mc P_I}'(\mc H)$, at least $(m+1)$ of $(n+h_i)$ are in $\mc P_I$, so
$S_{\mc P_I}'(\mc H)\subseteq S_{\mc P_I}(\mc H)$.
Summing the previous inequality for different values of $n$, we get
\begin{equation}\label{eq:bound-sum-of-weights}
\sum_{\substack{N < n \leq 2N \\ n \in S'_{\mc P_I}(\mc H)}} \Biggl(\sum_{i = 1}^k \one_\PP (n + h_i) u_-(\cos \theta_{n + h_i}) - m \Biggr) 
\leq k \sum_{\substack{N < n \leq 2N \\ n \in S_{\mc P_I}(\mc H)}} 1.
\end{equation}

Now suppose $n$ is such that $P^-\bigl(\prod_{i=1}^k (n + h_i)\bigr) \geq n^{c_1(k)}$ (e.g., this is true for $n \in S'_{\mc P_I}$). Then each $n + h_i$ has all prime factors bounded below by $q_1, \ldots, q_{r_i} \geq n^{c_1(k)}$, so that $n + h_i \geq q_1\cdots q_{r_i} \geq n^{r_i c_1(k)}$; this puts an upper bound of the number of prime factors $r_i$ of $n + h_i$, depending only on $c_1(k)$ and $\mc H$. This yields
\[
\left(\sum_{d_i \mid n + h_i \forall i} \lambda_{\mb{d}}\right)^2 \ll_{c_1(k), \mc H} \lambda_{\max}^2.
\]
But by our previous choice of $\lambda$ in \cref{prop:maynard-integral-asymptotics} we have
\[
\lambda_{\max}\ll_k\Biggl(\sup_{[0,1]^k} F\Biggr)(\log R)^k \ll_{F,k} (\log R)^k.
\]
Hence, we obtain
\[
1 \gg_{k, \mc H} \frac{1}{(\log R)^{2k}} \left(\sum_{d_i \mid n + h_i \forall i} \lambda_{\mb{d}}\right)^2
\]
for all $n$ such that $n + h_i$ has prime factors $\geq n^{c_1(k)}$. Combining this with our previous bound from \cref{eq:bound-sum-of-weights}, we have that
\begin{align*}
\sum_{\substack{N < n \leq 2N \\ n \in S_{\mc P_I}(\mc H)}} 1 
&\gg_{F,k,\mc H}
\frac{1}{(\log R)^{2k}} \sum_{\substack{N < n \leq 2N \\ n \in S'_{\mc P_I}(\mc H)}} \Biggl( \sum_{i = 1}^k \one_\PP (n + h_i) u_-(\cos \theta_{n + h_i}) - m \Biggr) \left(\sum_{d_i \mid n + h_i \forall i} \lambda_{\mb{d}}\right)^2 \\
&\geq \frac{1}{(\log R)^{2k}} \sum_{\substack{N < n \leq 2N \\ n \equiv v_0 \pmod{W} \\ P^-\bigl(\prod_{i=1}^k(n+h_i)\bigr)\ge n^{c_1(k)}}} \Biggl( \sum_{i = 1}^k \one_\PP (n + h_i) u_-(\cos \theta_{n + h_i}) - m \Biggr) \left(\sum_{d_i \mid n + h_i \forall i} \lambda_{\mb{d}}\right)^2 \\
&= \frac{1}{(\log R)^{2k}} \Bigl(\widehat{S_2}^+ - mS_1^+\Bigr).
\end{align*}
The second inequality above follows directly from the characterization of $S'_{\mc P_I}(\mc H)$, since the integers $N < n \leq 2N$ with $P^-\bigl(\prod_{i=1}^k(n+h_i)\bigr)\ge n^{c_1(k)}$ that do not lie in $S'_{\mc P_I}(\mc H)$ have a nonpositive contribution to the latter sum. Now we are in a position to apply \cref{prop:average-Pintz-condition}, which gives that
\[
\sum_{\substack{N < n \leq 2N \\ n \in S_{\mc P_I}(\mc H)}} 1 \gg_{F,k,\mc H} \frac{1}{(\log R)^{2k}} N(\log R)^k = \frac{N}{(\log R)^k}, 
\]
for our large enough (chosen) value of $k$. Recall that $R = N^{\theta/2 - \delta}$, so $\log R = (\theta/2 - \delta)\log N$ for some small $\delta$. Hence we for our choice of $F$ (which depends only on $k$) we can write
\[
\#\{n \leq x : n \in S_{\mc P_I}(\mc H)\} \gg_{k, \mc H, \mc P_I} x(\log x)^{-k},
\]
completing our proof.
\end{proof}

Now we can prove our main result on a combination of bounded gaps with Green--Tao theorem, generalizing \cref{thm:simplified-green-tao}.

\begin{theorem}\label{thm:green-tao-sato-tate}
Suppose that an interval $I \subseteq [-1, 1]$ is $\Sym^{\ell_{\max}}$-minorizable, for some $\ell_{\max} \geq 0$, by a polynomial with average $b_0$ against $\mu_{ST}$.
Let $\mc{H} = \{h_1,\ldots,h_k\}$ be an admissible set. 
Let $E/\QQ$ be a non-CM elliptic curve, and
suppose that $L(s,\Sym^\ell E)$ is automorphic for all $0 \le \ell \le \ell_{max}$. Define $\widetilde\theta := 1/\max(2,\ell_{\max}-1)$, and let $m$ be a positive integer with
\[
k\ge C\exp\biggl(\frac{2m}{b_0\widetilde\theta}\biggr),
\]
where $C > 0$ is an absolute constant. Then there exists an $(m+1)$-element subset $\{h_1',\ldots,h_{m+1}'\}$ of $\mc{H}$ with the following property: for every positive integer $t$, there exist infinitely many nontrivial $t$-term arithmetic progressions of integers $n$ such that $n+h_i'\in \mc{P}_I$ for all $1\le i\le m+1$.
\end{theorem}

\begin{proof}
By \cref{prop:Pintz-condition}, the set
\[S_{\mc{P}_I}(\mc{H}) := \biggl\{n\in\NN:\text{ at least }m+1\text{ of }n+h_i\text{ are in }\mc{P}_I, P^-\Biggl(\prod_{i=1}^k(n+h_i)\Biggr)\ge n^{c_1(k)}\biggr\}\]
satisfies the conditions in \cref{thm:Pintz}, hence there exist arithmetic progressions in $S_{\mc{P}_I}(\mc{H})$ of arbitrary length.

Let $n_1<n_2<\ldots<n_{M}$ be an arithmetic progression in $S_{\mc{P}_I}(\mc{H})$, where $M$ is a positive integer to be determined. By definition of $S_{\mc{P}_I}(\mc{H})$, for each $n_i$, there exists $h_{i_1},\ldots,h_{i_{m+1}} \in \mc H$ such that $n_i + h_{i_j} \in \mc P_I$ for all $1\leq j\leq m+1$. We associate $n_i$ with the set $\{h_{i_1},\ldots,h_{i_{m+1}}\}$. There are only finitely many such sets associated with the $n_i$, so we can think of the situation as having the $n_i$ labeled by finitely many ``colors''.

Since there are only finitely many colors, by van der Waerden's theorem, for any integer $t$, there exists sufficiently large $M=M(t,m,k)$ such that there is a monochromatic arithmetic progression $J\subseteq \{1,\ldots,M\}$ of length $t$. Then for all $j\in J$, $n_j$ is associated to the same set $\mc{H}_t$, which is of size $m+1$. Now, $\{n_j:j\in J\}$ is also an arithmetic progression, which satisfies $n_j+h\in\mc P_I$ for all $j\in J$ and $h\in\mc{H}_t$. Finally, choose $\mc{H}'$ be such that $\mc H'$ appears infinitely many times in the sequence $(\mc H_t)_{t\geq 1}$. Then $\mc H'$ satisfies \cref{thm:green-tao-sato-tate}, as desired.
\end{proof}

\begin{proof}[Proof of \cref{thm:simplified-green-tao}]
This is similar to the proof that \cref{thm:bounded-gaps-sato-tate} implies \cref{thm:simplified-bounded-gaps}.
\end{proof}

\section{Symmetric power \texorpdfstring{$L$}{L}-functions for elliptic curves}\label{sec:symmetric-powers}

It now remains to prove the prime number theorem and Bombieri--Vinogradov type estimate from \cref{sec:PNT-BV-sato-tate}. 
Towards this, we shift our attention to symmetric power $L$-functions for non-CM elliptic curves over $\QQ$.
Many of the formulas in this section are concisely stated in the case of squarefree conductor in \cite[Section~1]{RT17}.

Let $E$ be a non-CM elliptic curve over $\QQ$ with conductor $N_E$. 
The $\ell$th symmetric power $L$-function associated to $E$ is defined as
\[
L(s,\Sym^\ell E) := \prod_{p\nmid N_E} \prod_{j=0}^{\ell}(1-\alpha_p^{j}\beta_p^{\ell-j}p^{-s})^{-1} \prod_{p\mid N_E}\prod_{j = 0}^\ell(1-\gamma_{j,\Sym^\ell E}(p)p^{-s})^{-1},
\]
where $\alpha_p$ and $\beta_p$ are the local roots for $E$ at primes $p$ not dividing $N_E$; explicitly, 
$\alpha_p = \exp(i\theta_p)$ and $\beta_p = \exp(-i\theta_p)$.
At primes $p\mid N_E$, the Satake parameters $\gamma_{j,\Sym^\ell E}(p)$ are given in \cite[Appendix~A.3]{DGMP18} and have absolute value at most $1$.
When $p$ divides $N_E$, the coefficient of $p^{-s}$ in $L(s,\Sym^\ell E)$ is given by $U_\ell (\cos \theta_p)$, where $U_\ell$ is the $\ell$th Chebyshev polynomial of the second kind, defined in \cref{eq:chebyshevdefn}. Moreover, taking the logarithmic derivative, we find that the coefficient of $p^{-ms}$ in $-(L'/L)(s, \Sym^{\ell} E)$ is given by $U_{\ell}(\cos m\theta_p) \log p$ when $p\nmid N_E$.

Let $\chi$ be a Dirichlet character of modulus $q$. The twisted $L$-functions $L(s, \Sym^\ell E\otimes\chi)$ are defined in the usual way via Rankin--Selberg convolution: 
\[
L(s, \Sym^\ell E\otimes\chi) = \prod_{p\nmid qN_E}\prod_{j = 0}^\ell (1 - \alpha_p^j\beta_p^{\ell - j} \chi(p) p^{-s})^{-1} \prod_{p\mid qN_E} \prod_{j=0}^\ell (1-\gamma_{j,\Sym^\ell E\otimes \chi}(p)p^{-s})^{-1}.
\]
In the above formula, we shall denote the Euler factor corresponding to a general prime $p$ by
\[L_p(s,\Sym^\ell E\otimes \chi) := \prod_{j=0}^\ell (1-\alpha_{j,\Sym^\ell E\otimes \chi}(p)p^{-s})^{-1}.
\]

It is conjectured that for all values of $\ell$, $L(s, \Sym^\ell E)$ can be associated to a cuspidal automorphic representation, that is, the $L$-function is automorphic. 
This is known in general for $\ell\leq 6$, and for $\ell\leq 8$ for suitable fields, including $\QQ$, as mentioned in \cref{sec:overview-of-argument}. 
Knowing automorphy for a given value of $\ell$ implies the following conjecture for that $\ell$, adapted from \cite[Conjecture~1.1]{RT17}. 

\begin{conjecture}\label{conj:gamma-factor}
Let $\ell \ge 1$ be an integer, $E/\QQ$ a non-CM elliptic curve, and $\chi$ a primitive Dirichlet character of modulus $q$. Denote the conductors of $L(\Sym^\ell E,s)$ and $L(s, \Sym^\ell E\otimes\chi)$ by $A$ and $A_\chi$, respectively. Then the function
\[
\Lambda(s, \Sym^\ell E \otimes \chi) := A_\chi^{s/2}L_\infty(s,\Sym^\ell E\otimes \chi)L(s,\Sym^\ell E\otimes \chi)
\]
is entire of order one, and there exists a complex number $w_{\Sym^\ell E\otimes \chi}$ of modulus 1 such that
\begin{equation}\label{eq:completed-functional-equation}
    \Lambda(s,\Sym^\ell E\otimes \chi) = w_{\Sym^\ell E\otimes \chi}\Lambda(1-s,\Sym^\ell E\otimes \overline{\chi}),
\end{equation}
where
\begin{itemize}
    \item $A$ is a positive integer satisfying $A\le N_E^{O(\ell)}$ and $\rad(A) \mid N_E$;
    \item $A_\chi$ is a positive integer satisfying $A_\chi \leq Aq^{\ell+1}$, $\rad(A_\chi)\mid qN_E$, and $A_\chi = A_{\overline{\chi}}$;
    \item The $L_\infty$ factor is
    \[L_\infty(s, \Sym^\ell E\otimes \chi) =\pi^{-(\ell+1)s/2} \prod_{j=0}^{\ell} \Gamma\biggl(\frac{s+\kappa_{ j,\Sym^\ell E\otimes \chi}}{2}\biggr),\]
    for certain constants $\kappa_{j,\Sym^\ell E\otimes \chi}$ satisfying $0 \le \kappa_{j,\Sym^\ell E\otimes \chi} \le 2+\ell/2$. 
\end{itemize}
\end{conjecture}

\begin{remark}
Assuming automorphy of $L(s, \Sym^\ell E)$, the bounds on $A$ and $A_\chi$ from \cref{conj:gamma-factor} follow using \cite[Section~5]{Rou07} and \cite[Appendix~A]{DGMP18}. The local roots $\gamma_{j,\Sym^\ell E\otimes \chi}(p)$ at primes $p\mid qN_E$ have magnitude at most $1$, since the local roots of both $L(s, \Sym^\ell E)$ and $L(s, \chi)$ satisfy the same bounds (see e.g. \cite[(2.6)]{ST19}).
\end{remark}
\begin{remark}
As mentioned earlier, \cref{conj:gamma-factor} is known for $0\le\ell\le 8$.
\end{remark}
\section{Phragm\'en--Lindel\"of and \texorpdfstring{$L$}{L}-function bounds}\label{sec:parameter-bounds}
Here we shall derive several bounds on quantities related to symmetric power $L$-functions, which will be essential in the proofs of \cref{sec:zero-free-siegel} and \cref{sec:bombieri}. Throughout this section, fix $\ell \ge 1$ and assume that $L(s, \Sym^\ell E)$ is automorphic; hence \cref{conj:gamma-factor} holds true in our case.
\subsection{Bounds on coefficients of \texorpdfstring{$L$}{L}-functions}
As described in \cref{conj:gamma-factor}, we have a completed function $\Lambda(s, \Sym^\ell E\otimes\chi) = A_\chi^{s/2}L_\infty(s, \Sym^\ell E\otimes\chi)L(s, \Sym^\ell E\otimes\chi)$ satisfying the functional equation \cref{eq:completed-functional-equation}.
Rearranging this gives
\[
L(s, \Sym^\ell E\otimes\chi) = \eta(s, \chi)L(1 - s, \Sym^\ell E\otimes\overline{\chi}),
\]
where $\eta(s,\chi)$ is some combination of gamma factors. Since $A_\chi = A_{\overline{\chi}}$, we have
\[
\eta(s, \chi) = w_{\Sym^\ell E\otimes \chi} A_\chi^{\frac{1}{2} - s}\frac{L_\infty(1 - s, \Sym^\ell E\otimes\overline{\chi})}{L_\infty(s, \Sym^\ell E\otimes\chi)}.
\]

Since the local Euler factors of $L(s,\Sym^\ell E)$ at primes $p\mid N_E$ do not necessarily split into desired forms, it will be more convenient to work with the unramified part of the symmetric power $L$-function, defined as follows:
\begin{align*}
L_\unr(s,\Sym^\ell E) := \prod_{p\nmid N_E} \prod_{j=0}^\ell (1-\alpha_p^j\beta_p^{\ell-j}p^{-s})^{-1}=: \sum_{n\geq 1}a_nn^{-s}.
\end{align*}
We also define the unramified part for the twisted symmetric power $L$-function as
\begin{equation}\label{eq:definition-of-L-unr}
L_\unr(s,\Sym^\ell E\otimes\chi) := \prod_{p\nmid qN_E} \prod_{j=0}^\ell (1-\alpha_p^j\beta_p^{\ell-j}\chi(p)p^{-s})^{-1} = \prod_{p\nmid N_E}\prod_{j=0}^\ell(1-\alpha_p^j\beta_p^{\ell-j}\chi(p)p^{-s})^{-1},
\end{equation}
using that $\chi(p) = 0$ for $p|q$. Thus we find
\[
L_\unr(s, \Sym^\ell E\otimes\chi) = \sum_{n\ge 1}a_n\chi(n)n^{-s}.
\]
Then we can define
\begin{align*}
\sum_{n\ge 1}b_n\chi(n)n^{-s}:=\frac{1}{L_\unr(s, \Sym^\ell E\otimes\chi)},\qquad
\sum_{n\ge 1}\Lambda(n)c_n\chi(n)n^{-s}:=-\frac{L_\unr'}{L_\unr}(s, \Sym^\ell E\otimes\chi),
\end{align*}
noting that $b_n$ and $c_n$ are independent of $\chi$. The functional equation for $L_\unr$ is
\[L_\unr(s,\Sym^\ell E\otimes \chi) = \wt{\eta}(s,\chi)L_\unr(1-s,\Sym^\ell E\otimes \overline{\chi}),
\]
where we define
\begin{align*}\wt{\eta}(s,\chi):=&\  w_{\Sym^\ell E\otimes \chi} A_\chi^{\frac{1}{2} - s}\frac{L_\infty(1 - s, \Sym^\ell E\otimes\overline{\chi})}{L_\infty(s, \Sym^\ell E\otimes\chi)}\prod_{p\mid qN_E} \frac{L_p(1-s,\Sym^\ell E\otimes \overline{\chi})}{L_p(s,\Sym^\ell E\otimes \chi)}\\
=&\ \eta(s,\chi)\prod_{p\mid qN_E} \prod_{j=0}^\ell \frac{1-\gamma_{j,\Sym^\ell E\otimes \chi}(p)p^{-s}}{1-\gamma_{j,\Sym^\ell E\otimes \overline{\chi}}(p)p^{s-1}}.
\end{align*}

It will be useful to find upper bounds for $a_n,b_n,$ and $c_n$. 
From the form of the local roots of $L(s, \Sym^\ell E \otimes \chi)$, one easily obtains
\[
|a_n|\le \tau_{\ell+1}(n), \qquad\qquad |b_n|\le \tau_2(n)^{\ell+1}, \qquad\qquad
|c_n|\le\ell + 1,
\]
where $\tau_a(n)$ is the number of ways to write $n$ as an ordered product of $a$ positive integers. We will need to bound partial sums of powers of these generalized divisor functions. We have the following lemma, due to Norton \cite{N92}.
\begin{lemma}[{\cite[Lemma~2.5]{N92}}]\label{lem:divisor-power-sum}
Let $h$ be a real-valued multiplicative function such that $h(p^{a-1})\le h(p^a)$ for all primes $p$ and $a\ge 1$. Then for $x\ge 1$ and $\sigma\ge 1$, we have
\[\sum_{n\le x}h(n)\le x^\sigma\prod_{p\le x}\Biggl(1+\sum_{a = 1}^{\lfloor\log_px\rfloor}\frac{h(p^a)-h(p^{a-1})}{p^{a\sigma}}\Biggr).\]
\end{lemma}
By using $\sigma = 1$ and the Mertens estimate on $\sum_{p\le x}1/p$, it is easy to use this result to bound sums of products of generalized divisor functions as expressions of the type $x(\log x)^B$. Combining this with partial summation will allow us to estimate very general number-theoretic sums.

\subsection{Bounds on values of \texorpdfstring{$L$}{L}-functions and their derivatives}\label{subsec:L-function-bounds}

We will need bounds for the quantities $\widetilde{\eta}(\sigma+it,\chi), L(\sigma+it, \Sym^\ell E\otimes \chi),$ and $L_\unr(\sigma+it, \Sym^\ell E\otimes \chi)$, as well as their derivatives, for some suitable ranges of $\sigma$. We shall bound $\eta$ and $L$ first. Noting that $\widetilde{\eta}$ and $\eta$, as well as $L_\unr$ and $L$, only differ by some (Euler) factors at primes $p\mid qN_E$, we can immediately obtain the bounds of $\widetilde{\eta}$ and $L_\unr$ from their counterparts. 

Using Stirling's formula and the explicit form of the gamma factors from \cref{conj:gamma-factor}, we find for $\sigma\in (-1/4,1/2)$ that
\begin{align*}
|\eta(\sigma+it,\chi)|\ll_{\ell} (A_{\chi}(|t|+2)^{\ell+1})^{\frac{1}{2}-\sigma}.
\end{align*}
Now we bound the log derivative of $\eta$ for $\sigma \in (-1/4,1/2)$. A similar method yields the estimate
\begin{align*}
\biggl|\frac{\eta'}{\eta}(\sigma+it,\chi)\biggr|\ll_{\ell}\log (A_{\chi}(|t|+2)^{\ell}).
\end{align*}
Combining the two inequalities gives a bound for $\eta'(\sigma+it).$

To obtain some estimates for the sizes of the relevant $L$-functions on vertical strips, we shall apply the Phragm\'en--Lindel\"of theorem. 
The following version is from Rademacher \cite{Rad60}.
\begin{theorem}[Phragm\'en--Lindel\"of]\label{thm:phragmen-lindelof}
Write $s = \sigma + it$. Suppose $h$ is a holomorphic function of finite order in  the strip $c\le\sigma\le d$. Suppose further there are constants $C, D, Q, \alpha, \beta$ such that
\begin{align*}
|h(c + it)|\le C|Q + c + it|^\alpha,\qquad 
|h(d + it)|\le D|Q + d + it|^\beta.
\end{align*}
Then for all $c\le\sigma\le d$ we have that
\[|h(s)|\le (C|Q + s|)^{\frac{d-\sigma}{d-c}}(D|Q + s|^\beta)^{\frac{\sigma-c}{d-c}}.\]
\end{theorem}
We now bound $L(s,\Sym^\ell E\otimes \chi)$ and $L'(s, \Sym^\ell E\otimes \chi)$ in the critical strip $0\le\sigma\le 1$.
\begin{proposition}\label{lem:L-L'-sizes}
We have for $0\leq \sigma \leq 1$ that
\begin{align*}
|L(\sigma + it, \Sym^\ell E\otimes\chi)|&\ll_{\ell} (A_\chi(|t| + 2)^{\ell + 1})^{\frac{1-\sigma}{2}}(\log [A_\chi(|t| + 2)^{\ell+1}])^{\ell+1},\\
|L'(\sigma + it, \Sym^\ell E\otimes \chi)|&\ll_{\ell} (A_\chi(|t| + 2)^{\ell + 1})^{\frac{1-\sigma}{2}}(\log [A_\chi(|t| + 2)^{\ell+1}])^{2\ell+2}.
\end{align*}
\end{proposition}

\begin{proof}
From splitting the Euler product into $\ell + 1$ terms and expanding, we find that
\[|L(1 + \eps + it, \Sym^\ell E\otimes \chi)|\le\zeta(1+\eps)^{\ell+1}.\]
Using the functional equation and our bounds on $|\eta(s, \chi)|$, we see that
\[|L(-\eps + it, \Sym^\ell E\otimes \chi)|\ll_{\ell} (A_\chi(|t| + 2)^{\ell+1})^{\frac{1}{2}+\eps}\zeta(1+\eps)^{\ell+1}\]
if $0 < \eps < 1/2$. Note that $L(s, \Sym^\ell E\otimes\chi)$ does not have a pole, even when $\chi$ is trivial, as $\ell\ge 1$ and it is order $1$ by assumption of automorphy. Thus by \cref{thm:phragmen-lindelof}, we see that
\[|L(\sigma + it,\Sym^\ell E\otimes \chi)|\ll_{\ell}\zeta(1+\eps)^{\ell+1}(A_\chi(|t|+2)^{\ell+1})^{\frac{1+\eps-\sigma}{2}}.\]
Choosing $\eps = 1/{\log[A_\chi(|t|+2)^{\ell+1}]}$ and using $\zeta(1+\eps)\ll \eps^{-1}$, we conclude that
\[|L(\sigma + it,\Sym^\ell E\otimes \chi)|\ll_{\ell}(A_\chi(|t|+2)^{\ell+1})^{\frac{1-\sigma}{2}}(\log[A_\chi(|t|+2)^{\ell+1}])^{\ell+1},\]
as desired.

Now we do the same for $L'(s, \Sym^\ell E\otimes \chi)$. By considering $-(L'/L)(s, \Sym^\ell E\otimes \chi)$, using the functional equation, and multiplying, we can bound $L'(\sigma + it,\Sym^\ell E\otimes \chi)$ on $\sigma = -\eps$ and $\sigma = 1 + \eps$, and then apply the Phragm\'en--Lindel\"of theorem. The fact that $-(\zeta'/\zeta)(1 + \eps)\ll\eps^{-1}$ introduces an extra factor of $\ll_{\ell}\log[A_\chi(|t|+2)^{\ell+1}]^{\ell+1}$.
\end{proof}
\begin{remark}
The above bounds on $L$ and $\eta$ allow us to compute bounds for $L_\unr$ and $L_\unr'$, as well as $\widetilde\eta$ and $\widetilde\eta'$. By bounding the Euler factors at primes $p\mid qN_E$, if $q\le Q$ we see for any $\eps'>0$ that
\begin{align*}
|L_\unr(s,\Sym^\ell E\otimes \chi)|&\ll_{\ell,\eps'}(QA)^{\eps'}|L(s,\Sym^\ell E\otimes\chi)|,\\
|\widetilde\eta(s,\chi)|&\ll_{\ell,\eps'}(QA)^{\eps'}|\eta(s,\chi)|,
\end{align*}
so long as $-1/\log(AQ)\ll_\ell\sigma\le 1/2$. Similar results can be obtained for the derivatives. 
\end{remark}

\section{The Siegel--Walfisz theorem for symmetric powers}\label{sec:zero-free-siegel}

Following the notation from \cite[Section~5]{IK04}, define
\[
\psi(x, \Sym^{\ell} E \otimes \chi):= \sum_{n\le x}\Lambda_{\Sym^{\ell} E \otimes \chi}(n),
\]
where $\Lambda_{\Sym^{\ell} E \otimes \chi}$ is the coefficient of $n^{-s}$ in $-(L'/L)(s, \Sym^{\ell} E \otimes \chi)$, equal to
\begin{equation}
\Lambda_{\Sym^\ell E\otimes\chi}(p^m) = \log(p)\sum_{j = 0}^\ell\alpha_{j,\Sym^\ell E\otimes}(p)^m
\end{equation}
at prime powers and $0$ elsewhere. Since, as noted, the local roots have magnitude bounded by $1$, we have the bound
\begin{equation}\label{eq:lambda-coefficient-bound}
|\Lambda_{\Sym^\ell E\otimes\chi}(n)|\le (\ell+1)\Lambda(n),
\end{equation}
where $\Lambda(n)$ is the usual von Mangoldt function.

\begin{lemma}\label{lem:siegel-zero-free}
Let $E/\QQ$ be a non-CM elliptic curve, and let $\chi$ be a primitive Dirichlet character of modulus $q$. If we assume automorphy of $L(s, \Sym^{\ell} E)$, where $s = \sigma + it$, then we have $L(s, \Sym^{\ell} \otimes \chi)\neq 0$ whenever
\[
\sigma \ge 1 - \frac{c_{\ell,E}}{\log(q(|t|+3)},
\]
except possibly for one real Siegel zero $\beta_{\chi}$ when $\chi$ is real. This zero satisfies
\[
\beta_{\chi}\le 1 - \frac{c(\ell, E, \eps)}{q^{\eps}}
\]
for any fixed $\eps>0$.
\end{lemma}
\begin{proof}
The zero-free region holds by \cite[Theorem~A.1]{HB19} for $\pi = \chi$ and $\pi'$ the cuspidal automorphic representation associated to $\Sym^{\ell} E$. To bound $\beta_\chi$, note that we have
\[
L(1, \Sym^{\ell} E \otimes \chi)\gg_{\ell, E, \eps} q^{-\eps}
\]
for all $\eps>0$ by \cite[Theorem~2.3.2]{Mo99} for the $L$-function $L(s, \Sym^{\ell} E)$ in the $\mc{L}^{*}$-class of irreducible automorphic representations of $\GL_{\ell +1}(\AA_{\QQ})$. The hypotheses of the theorem hold (since the local roots of $L(s, \Sym^{\ell} E)$ are bounded in magnitude by $1$, which implies the same for $L(s, \Sym^{\ell} E \otimes \Sym^{\ell} E)$; see, for example, \cite[(A.12)]{RS96}). This lower bound, along with the upper bound for $|L'(s, \Sym^{\ell} E \otimes \chi)|$ established in \cref{lem:L-L'-sizes} (and the inequality $A_{\chi} \le A q^{\ell+1}$ from \cref{sec:parameter-bounds}), implies the desired inequality for $\beta_{\chi}$ by the mean value theorem (as in, for example, \cite[Ch.~21]{DA82}).
\end{proof}
We have the following analogue of the Siegel--Walfisz theorem, which estimates $\psi(x, \Sym^{\ell} E \otimes \chi)$.
\begin{theorem}\label{thm:siegel-walfisz}
Let $E/\QQ$ be a non-CM elliptic curve, and $\chi$ be a primitive Dirichlet character of modulus $q$. For any $\ell\ge 1$ and any $B > 0$, assuming automorphy of $L(s, \Sym^\ell E)$, there exists a constant $c_{B,E,\ell}>0$ such that
\[
\psi(x,\Sym^\ell E\otimes\chi) = O_{B,E,\ell}\Bigl(x\exp\Bigl(-c_{B,E,\ell}\sqrt{\log x}\Bigr)\Bigr)
\]
in the range $q \le (\log x)^B$.
\end{theorem}

\begin{proof}
Given our zero-free region from \cref{lem:siegel-zero-free}, \cite[Theorem~5.13]{IK04} provides an analogue of the prime number theorem depending on $\beta_\chi$. Adapting the cited theorem for $\Sym^\ell E \otimes \chi$ (whose $L$-function has no pole at $s = 1$ by \cite[p.~136]{IK04}), we get
\begin{equation}\label{eq:preliminary-pnt}
\psi(x, \Sym^\ell E \otimes\chi) = -\frac{x^{\beta_\chi}}{\beta_\chi} + O_{\ell, E}\biggl(x \exp \biggl(\frac{-c_{\ell, E} \log x}{\sqrt{\log x} + \log q}\biggr) (\log xq)^4\biggr)
\end{equation}
for $x \ge 1$, where the term depending on $\beta_\chi$ should be neglected if there is no Siegel zero. Note that a stronger version of the bound on $\Lambda_{\Sym^{\ell} E \otimes \chi}(n)$ required by \cite[Theorem~5.13]{IK04} is true in our case, namely \cref{eq:lambda-coefficient-bound}.
Combining the domain restriction $q\le (\log x)^B$ with the bound on $\beta_\chi$ from \cref{lem:siegel-zero-free} implies the desired result.
\end{proof}
In particular, by taking $\chi = 1$ in our analogue of the Siegel--Walfisz theorem and using partial summation, we can establish \cref{thm:Ul-prime-number-theorem}.
\begin{proof}[Proof of  \cref{thm:Ul-prime-number-theorem}]
Taking $\chi$ to be the trivial character in \cref{thm:siegel-walfisz}, we conclude that
\[
\psi(x, \Sym^{\ell} E) = O_{E,\ell}\bigl(x \exp\bigl(-c_{E,\ell} \sqrt{\log x}\bigr)\bigr).
\]
We rewrite the left hand side as  
\[
\psi(x, \Sym^{\ell} E) = \sum_{p^{m}\le x} \Lambda_{\Sym^{\ell} E}(p^m)= \sum_{\substack{p^{m} \le x\\ p\nmid N_E}} U_{\ell}(\cos m \theta_p) \log p+ \sum_{\substack{p^{m} \le x\\ p\mid N_E}} \Lambda_{\Sym^{\ell} E}(p^m).
\]
Using \cref{eq:lambda-coefficient-bound}, the second term in the right-hand side above is bounded by 
\[
(\ell + 1) \sum_{p \mid N_E} (\log_p x)(\log p) \ll_{E,\ell} \log x,
\]
which is absorbed by the error term in the estimate of $\psi(x, \Sym^\ell E)$. By a standard partial summation argument, we then establish that
\[
\sum_{\substack{p\le x\\p\nmid N_E}} U_{\ell}(\cos \theta_p) = O_{E,\ell}\bigl(x \exp\bigl(-c_{E,\ell}'\sqrt{\log x}\bigr)\bigr).
\]
Finally, by adding back the terms corresponding to finitely many primes $p$ dividing $N_E$, as well as using $|U_\ell(\cos \theta_p)| \le \ell + 1$, we may remove the restriction $p \nmid N_E$ and conclude our proof.
\end{proof}

\section{Weighted Bombieri--Vinogradov for Sato--Tate}\label{sec:bombieri}
In this section we prove our Bombieri--Vinogradov type estimate in the sense of \cref{conj:bombieri-sato-tate-final}. 
The structure of our approach is similar to that of Murty and Murty \cite{MM87}. First, we reduce the statement of the theorem to a similar estimate on a certain sum involving the twisted unramified part of the symmetric power $L$-function $L_\unr(s,\Sym^\ell E\otimes \chi)$, defined in \cref{eq:definition-of-L-unr}. 
After this, we use Gallagher's method to split the sum into three parts.
Within the sum, some parts can be directly estimated using large sieve inequalities. 
The parts involving $L_\unr(s,\Sym^\ell E\otimes \chi)$ and $L'_\unr(s,\Sym^\ell E\otimes \chi)$ will be further decomposed using Ramachandra's method and estimated separately. 
Finally, we shall choose suitable parameters and combine all the estimates to conclude the proof of Bombieri--Vinogradov in the setting of Sato--Tate.
\subsection{Initial reductions}\label{subsec:initial-reductions}

In what follows, we will sometimes drop the dependence on $E$ and $\ell$ when it is clear. In order to establish \cref{conj:bombieri-sato-tate-final}, we reduce the inequality to an analogous result for log-weighted Chebyshev functions. We define
\begin{align*}
\psi_{r, \ell, E}(x; q, a) &:= \frac{1}{r!}\sum_{\substack{p^m\le x\\p^m\equiv a\pmod{q}}}U_\ell(\cos m\theta_p)(\log p)\biggl(\log\frac{x}{p^m}\biggr)^r,\\
\psi_{r, \ell, E}(x, \chi) &:= \frac{1}{r!}\sum_{p^m\le x}U_\ell(\cos m\theta_p)(\log p)\biggl(\log\frac{x}{p^m}\biggr)^r\chi(p^m).
\end{align*}
We see that these are related by the orthogonality relation
\[\psi_{r,\ell, E}(x; q, a) = \frac{1}{\varphi(q)}\sum_\chi\overline{\chi}(a)\psi_{r,\ell,E}(x, \chi).\]
First we reduce \cref{conj:bombieri-sato-tate-final} to a traditional Chebyshev function version of Bombieri--Vinogradov.
\begin{proposition}\label{conj:bombieri-sato-tate-penultimate}
Fix $E$ and suppose we have automorphy of the $\ell$th symmetric power of $E$, where $\ell\ge 1$. Then for any $0 < \theta < 1/{\max(2, \ell-1)}$, and for all $B > 0$, we have that
\[\sum_{\substack{q\le x^\theta}}\sup_{\substack{(a, q) = 1\\y\le x}}|\psi_{0, \ell, E}(x; q, a)|\ll_{B, E, \ell}x(\log x)^{-B}.\]
\end{proposition}
\begin{proof}[Proof that  \cref{conj:bombieri-sato-tate-penultimate} implies \cref{conj:bombieri-sato-tate-final}]
A standard exercise in partial summation.
\end{proof}
Note that we do not care about the dependence of this inequality on $B, E$ and $\ell$, and it will in fact end up being ineffective due to the use of Siegel--Walfisz for Sato--Tate from \cref{thm:siegel-walfisz}. 
We show that if $r\ge 0$ is fixed depending on $\theta, \ell$, and if we have a Bombieri--Vinogradov type bound for $\psi_{r, \ell}$, then we can deduce the above \cref{conj:bombieri-sato-tate-penultimate}.
\begin{proposition}\label{conj:bombieri-sato-tate-antepenultimate}
Fix $E$ and suppose we have automorphy of the $\ell$th symmetric power of $E$, where $\ell\ge 1$. 
Fix any $0 < \theta < 1/{\max(2, \ell-1)}$.
Then there exists an $r = r(\theta, \ell)\ge 0$ so that for all $B > 0$ we have that
\[\sum_{\substack{q\le x^\theta}}\sup_{\substack{(a, q) = 1\\y\le x}}|\psi_{r, \ell, E}(x; q, a)|\ll_{B, E, \ell}x(\log x)^{-B}.\]
\end{proposition}
\begin{proof}[Proof that \cref{conj:bombieri-sato-tate-antepenultimate} implies \cref{conj:bombieri-sato-tate-penultimate}]
The argument is similar to the analogous result in \cite{MM87}, with some alterations due to positivity issues. In order to work with a nonnegative functions, we look at $g_r(x; q, a) = \psi_{r, \ell, E}(x; q, a) + (\ell+1)\psi_r(x; q, a)$, where we define
\[\psi_r(x; q, a) := \sum_{\substack{p^m\le x\\p^m\equiv a\pmod{q}}}(\log p)\biggl(\log\frac{x}{p^m}\biggr)^r,\]
and then relate $g_{r+1}$ and $g_r$ as in \cite{MM87}. 
To finish the argument, we use the regular Bombieri--Vinogradov theorem to show that the contribution of the $\psi_r(x; q, a)$ terms does not greatly affect the estimates involved.
\end{proof}
The remainder of this section is dedicated to proving \cref{conj:bombieri-sato-tate-antepenultimate}. From now on we will take $E$ and $\ell\ge 1$ to be fixed. We will choose $r = r(\theta, \ell)$ later. With foresight, we will take $r = \frac{\ell+1}{2} + \frac{(\ell+1)^2}{2\eps}$ for some small $0 < \eps < \frac{1}{\max(2, \ell - 1)} - \theta$.

Let $(C)$ denote the contour $C + it$ for $t\in (-\infty, \infty)$. 
Using Mellin inversion, we have for any $C > 1$ that
\begin{equation*}\label{eq:perron-integral-r-chi}
\psi_{r,\ell,E}(x, \chi) = \frac{1}{2\pi i}\int_{(C)}-\frac{L_\unr'}{L_\unr}(s,\Sym^\ell E\otimes  \chi)\frac{x^s}{s^{r+1}}\,ds + O(\ell(\log x)^{r+1}(\log qN_E)),
\end{equation*}
the error coming from the contribution of the primes $p$ dividing $qN_E$.

\subsection{Gallagher's method and the large sieve}\label{subsec:gallagher}
We use the version of the large sieve of Gallagher \cite{G70}. From now on, $\sum_\chi$ denotes a sum over all Dirichlet characters $\pmod{q}$ and $\sum_\chi^*$ a sum over all primitive characters $\pmod{q}$.
\begin{lemma}[Large sieve]
If $\sum_n |A_n|<\infty$ and $T\geq 1$, we have that
\[\sum_{q\leq Q}{\sum_{\chi}}^*\int_{-T}^T\Biggl|\sum_{n=1}^\infty A_n \chi(n)n^{it}\Biggr|^2 dt \ll \sum_{n=1}^\infty |A_n|^2(n+Q^2 T).\]
\end{lemma}
We will use this to estimate certain sums related to Bombieri--Vinogradov. They will be used to handle the averaging and cancellation that is expected in the regime of large $q$ values. Write the following partial sums, defined for $\Re s > 1$:
\begin{equation}
\begin{split}
F_z(s,\chi) &:= \sum_{n\leq z}\Lambda(n)\chi(n)c_nn^{-s},\\
G_z(s,\chi) &:= \sum_{n>z} \Lambda(n)\chi(n)c_nn^{-s} = -\frac{L_\unr'}{L_\unr}(s,\Sym^\ell E\otimes \chi) - F_z(s,\chi),\\
M_z(s,\chi) &:= \sum_{n\leq z}\chi(n)b_n n^{-s}.
\end{split}
\end{equation}
Then we can write
\[-\frac{L_\unr'}{L_\unr}(s,\Sym^\ell E\otimes \chi) = G_z(1-L_\unr(\chi) M_z)+ F_z(1-L_\unr(\chi) M_z) - L'_\unr(\chi) M_z,\]
where we are suppressing the $\chi$ dependence. Hence, for $C = 1 + (\log x)^{-1}$ and $r > \frac{\ell+1}{4}$, we have 
\begin{align}\label{eq:gallagher-identity}
\frac{1}{2\pi i}\int_{(C)}& -\frac{L_\unr'}{L_\unr}(s,\Sym^\ell E\otimes \chi)\frac{x^s}{s^{r+1}}ds \nonumber\\
=& \frac{1}{2\pi i}\int_{(C)}G_z(1 - L_\unr(\chi) M_z)\frac{x^s}{s^{r+1}}\,ds \nonumber\\
&\quad + \frac{1}{2\pi i}\int_{(\frac{1}{2})}F_z(1 - L_\unr(\chi) M_z)\frac{x^s}{s^{r+1}}\,ds\nonumber +\frac{1}{2\pi i}\int_{(\frac{1}{2})}L_\unr'(\chi)M_z\frac{x^s}{s^{r+1}}\,ds\nonumber\\
\ll & x\int_{(C)}(|G_z|^2 + |1 - L_\unr(\chi) M_z|^2)\frac{1}{|s|^{r+1}}\,|ds|\nonumber\\
&\quad+ x^{\frac{1}{2}}\int_{(\frac{1}{2})}(1+|F_z|^2+|M_z|^2+|F_zM_z|^2+|L_\unr(\chi)|^2+|L_\unr'(\chi)|^2)\frac{1}{|s|^{r+1}}\,|ds|, 
\end{align}
where shifting the contour is acceptable since $F_z,M_z$ are holomorphic and $r > \frac{\ell + 1}{4}$. We thus see that the right side is an upper bound for $\sup_{y\le x}|\psi_{r,\ell}(y, \chi)|$, with an error of $O(\ell(\log x)^{r+1}(\log qN_E))$.

Using the large sieve, \cref{lem:divisor-power-sum}, bounds on $a_n,b_n,c_n$, and partial summation to bound the resulting sums yields
\begin{equation}\label{eq:Gz-LMz-FzMz-terms}
\begin{split}
\sum_{q\leq Q}{\sum_{\chi}}^*\int_{(C)} |G_z|^2 \frac{|ds|}{|s|^{r+1}} &\ll (\ell+1)^2(\log x)^3 \biggl(1+\frac{Q^2}{z}\biggr),\\
\sum_{q\leq Q}{\sum_{\chi}}^* \int_{(C)}|1-L_\unr(\chi) M_z|^2\frac{|ds|}{|s|^{r+1}}
&\ll_\ell (\log x)^{(\ell + 2)^22^{2\ell + 2}}\biggl(1+\frac{Q^2}{z}\biggr),\\
\sum_{q\leq Q}{\sum_{\chi}}^* \int_{(\frac{1}{2})}(1+|F_z|^2+|M_z|^2+|F_zM_z|^2 )\frac{|ds|}{|s|^{r+1}}&\ll_\ell (Q^2 + z^2)(\log z)^{2^{2\ell + 4} + 2}.
\end{split}
\end{equation}

\subsection{Ramachandra's method for mean square estimates}\label{subsec:ramachandra}
It remains to bound
\[\sum_{q\leq Q}{\sum_{\chi}}^*\int_{(\frac{1}{2})}(|L_\unr(\chi)|^2+|L_\unr'(\chi)|^2)\frac{|ds|}{|s|^{r+1}}.\]
We use a method of Ramachandra \cite{Ram74}. We first estimate
\[\sum_{q\leq Q}{\sum_{\chi}}^*\int_{(\frac{1}{2})}|L_\unr(\chi)|^2\frac{|ds|}{|s|^{r+1}}.\]
The key identity is the following.
\begin{equation}\label{eq:twisted-L-identity}
L_\unr(s,\Sym^\ell E\otimes \chi) = \sum_{n=1}^\infty a_n\chi(n)e^{-\frac{n}{U}}n^{-s} - \frac{1}{2\pi i}\int_{(C_1)}L_\unr(s+w,\Sym^\ell E\otimes \chi)U^w\Gamma(w)dw,
\end{equation}
where $U > 0$ and $C_1 = -\frac{1}{2}-\frac{1}{\log V}$, $V > 1$. With foresight, we choose
\begin{equation}\label{eq:choice-of-U-V}
U = V = (AQ^{\ell+1}T^{\ell+1})^{\frac{1}{2}}, \text{ where } T = Q^{\frac{\eps}{\ell + 1}}.
\end{equation}
We split the second term of \eqref{eq:twisted-L-identity} into two parts, based on whether $n>U$ or $n\leq U$. The part with $n\leq U$ is holomorphic, so we can move the contour to $(C_2)$ with $C_2 = -1/\log V$. This gives
\begin{align}\label{eq:damped-L-sum}
L_\unr(s,\Sym^\ell E\otimes \chi) &= \sum_{n=1}^\infty a_n\chi(n)e^{-\frac{n}{U}}n^{-s} - \frac{1}{2\pi i}\int_{(C_1)}\widetilde\eta(s+w,\chi)\sum_{n>U}a_n\overline{\chi}(n)n^{s+w-1}U^w\Gamma(w)\,dw\nonumber\\ & \qquad -\frac{1}{2\pi i}\int_{(C_2)}\widetilde\eta(s+w,\chi)\sum_{n\leq U}a_n\overline{\chi}(n)n^{s+w-1}U^w\Gamma(w)\,dw
\end{align}
when $\Re s = 1/2$, using the functional equation. We will truncate integrals to the height $T$, as chosen above. Hence, the truncated integral of $|L_\unr(\chi)|^2$ is, by Cauchy--Schwarz inequality,
\begin{align*}
\int_{\frac{1}{2}-iT}^{\frac{1}{2}+iT} |L_\unr(s,\Sym^\ell &E\otimes \chi)|^2 ds\\
&\ll \int_{-T}^T\Biggl|\sum_{n=1}^\infty a_n\chi(n)e^{-\frac{n}{U}}n^{-\frac{1}{2}+it}\Biggr|^2dt \\
&\quad +\int_{-T}^T\int_{(C_1)}\Biggl|\widetilde\eta\biggl(\frac{1}{2}+it+w,\chi\biggr)\sum_{n>U}a_n\overline{\chi}(n)n^{-\frac{1}{2}+it + w}U^w\Gamma(w)\Biggr|^2dwdt \\
&\quad +\int_{-T}^T\int_{(C_2)}\Biggl|\widetilde\eta\biggl(\frac{1}{2}+it+w,\chi\biggr)\sum_{n\leq U}a_n\overline{\chi}(n)n^{-\frac{1}{2}+it + w}U^w\Gamma(w)\Biggr|^2dwdt.
\end{align*}
Let the right-hand side be $S_1 + S_2 + S_3$. Sum over primitive characters $\chi \pmod q$ and $q\leq Q$, and write the resulting right-hand side as $\Sigma_1+\Sigma_2+\Sigma_3$ in the obvious way. We use the large sieve inequality and use bounds on $\widetilde\eta(s, \chi)$ (see the remark following \cref{lem:L-L'-sizes}). After that, using \cref{lem:divisor-power-sum} and partial summation, we find that
\begin{align*}
\Sigma_1&\ll_\ell (U + Q^2T)(\log U)^{(\ell + 1)^2},\\
\Sigma_2 &\ll_{\ell,\eps'} (U + Q^2T)(\log U)^{(\ell+1)^2}[(AQ)^{\eps'}],\\
\Sigma_3&\ll_{\ell,\eps'}\biggl(1+\frac{Q^2T}{U}\biggr)U(\log V)^2(\log U)^{(\ell+1)^2}[(AQ)^{\eps'}].
\end{align*}
We also need to bound the integral from $T$ to $\infty$. We have $r > \frac{\ell+1}{2}$, and thus we find
\begin{align*}
\int_T^\infty\biggl|L_\unr\biggl( \frac{1}{2}+it,\Sym^\ell E\otimes\chi\biggr)\biggr|^2\frac{dt}{\bigl|\frac{1}{2}+it\bigr|^{r+1}}
&\ll_{r,\ell,\eps'}\frac{(AQ^{\ell+1})^{\frac{1}{2}}Q^2(\log(AQ^{\ell+1}))^{2\ell+2}(\log T)^{2\ell+2}}{T^{r-\frac{\ell+1}{2}}}[(AQ)^{\eps'}].
\end{align*}
Combining the above estimates and choosing $U,V,T$ as in \cref{eq:choice-of-U-V}, we obtain
\begin{equation}\label{eq:L-term}
\sum_{q\le Q}{\sum_\chi}^*\int_{(\frac{1}{2})}|L_\unr(\chi)|^2\frac{|ds|}{|s|^{r+1}}\ll_{r,\ell,\eps,\eps'} (A^{\frac{1}{2}}Q^{\frac{\ell+1+\eps}{2}}+Q^{2+\frac{\eps}{\ell+1}}+A^{\frac{1}{2}}Q^2)[(AQ)^{\eps'}].
\end{equation}
Now we move on to estimating
\[\sum_{q\le Q}{\sum_\chi}^*\int_{(\frac{1}{2})}|L_\unr'(\chi)|^2\frac{|ds|}{|s|^{r+1}}.\]
The method is essentially the same. Differentiating \eqref{eq:damped-L-sum} and split now into five terms analogous to the above. The Phragm\'en--Lindel\"of estimate \cref{lem:L-L'-sizes} for $L'$ is bigger by a factor of $(\log[AQ^{\ell+1}(|t| + 2)^{\ell+1}])^{\ell+1}$; this contributes an extra factor of $(\log[AQ^{\ell+1}T^{\ell+1}])^{2\ell+2}$, when squared, to the integral from $T$ to $\infty$. Similar terms are added in to the other factors. So overall we have that
\begin{equation}\label{eq:L'-term}
\sum_{q\le Q}{\sum_\chi}^*\int_{(\frac{1}{2})}|L_\unr(\chi)'|^2\frac{|ds|}{|s|^{r+1}}\ll_{r,\ell,\eps,\eps'} (A^{\frac{1}{2}}Q^{\frac{\ell+1+\eps}{2}}+Q^{2+\frac{\eps}{\ell+1}}+A^{\frac{1}{2}}Q^2)[(AQ)^{\eps'}].
\end{equation}

\subsection{Concluding Bombieri--Vinogradov for Sato--Tate}\label{subsec:concluding-bombieri}
Combining \eqref{eq:Gz-LMz-FzMz-terms}, \eqref{eq:L-term}, and \eqref{eq:L'-term}, we obtain
\begin{equation}\label{eq:gallagher-final}
\begin{split}
\sum_{q\le Q}{\sum_\chi}^*\sup_{y\le x}|\psi_{r,\ell}(y, \chi)|&\ll_{\ell,\eps,\eps'} x(\log x)^{(\ell + 2)^22^{2\ell + 2}}\biggl(1+\frac{Q^2}{z}\biggr)+x^{\frac{1}{2}}(Q^2 + z^2)(\log z)^{2^{2\ell + 4} + 2}\\
&\qquad+ x^{\frac{1}{2}}[(AQ)^{\eps'}](A^{\frac{1}{2}}Q^{\frac{\ell+1+\eps}{2}}+Q^{2+\frac{\eps}{\ell+1}}+A^{\frac{1}{2}}Q^2).
\end{split}
\end{equation}
Now we take $z = Q(\log x)^\gamma$ with $\gamma > D + (\ell+1)^22^{2\ell+2}$. Then if $Q$ lies within the following range,
\[(\log x)^\gamma\le Q\le \min(x^{\frac{1}{2}}(\log x)^{-2\gamma - D - 2^{2\ell + 4} - 2}, x^{\frac{1}{2}-\eps-\eps'}, x^{\frac{1}{\ell-1} - \eps-2\eps'}),\]
we have that
\[\frac{1}{Q}\sum_{q\le Q}{\sum_\chi}^*\sup_{y\le x}|\psi_{r,\ell}(y, \chi)|\ll_{\ell,\eps,\eps',E}x(\log x)^{-D}.\]
(If $\ell = 1$ we can take $x^{\frac{1}{\ell-1}}=+\infty$ without issue.) Using \cref{thm:siegel-walfisz} and integration to convert from $\psi_{0,\ell}$ to $\psi_{r,\ell}$, we can prove this result for $Q\le (\log x)^\gamma$ as well. In our application of \cref{thm:siegel-walfisz}, we must note that the difference between $\psi_{0, \ell}$ and $\psi(\cdot,\Sym^\ell E\otimes\chi)$ due to primes $p\mid qN_E$ is bounded by $O(\ell(\log x)(\log qN_E))$ as in \cref{eq:perron-integral-r-chi}. Thus, the following result holds.
\begin{proposition}\label{prop:almost-bombieri}
Given hypotheses in \cref{conj:bombieri-sato-tate-final}, suppose $Q$ satisfies
\[1\le Q\le \min(x^{\frac{1}{2}}(\log x)^{-2\gamma - D - 2^{2\ell + 4} - 2}, x^{\frac{1}{2}-\eps-\eps'}, x^{\frac{1}{\ell-1} - \eps-2\eps'}).\]
Then we have that
\[\frac{1}{Q}\sum_{q\le Q}{\sum_\chi}^*\sup_{y\le x}|\psi_{r,\ell,E}(y, \chi)|\ll_{\ell,\eps,\eps',E,\gamma}x(\log x)^{-D}.\]
\end{proposition}
Now we prove a conversion result that allows us to transform this into a more manageable statement of Bombieri--Vinogradov type.
\begin{proposition}\label{prop:bombieri-conversion}
Fix $\theta < \frac{1}{2}$ and $\ell\ge 1$. 
Then for any $r$, we have that
\begin{align*}\sum_{\substack{q\leq x^\theta }}\sup_{\substack{(a,q)=1\\y\leq x}} |\psi_{r,\ell,E}(y;q,a)|&\ll (\log x)^2(\log \log x)^2\sup_{Q\leq x^\theta}\frac{1}{Q}\sum_{\substack{q\leq Q}}\sup_{y\le x}{\sum_{\chi}}^* |\psi_{r,\ell,E}(y,\chi)| \\
&\quad + O((\ell+1)x^\theta(\log x)^{r+2}).
\end{align*}
\end{proposition}
\begin{proof}
Similar to \cite[Proposition~1.8]{MM87}.
\end{proof}
\begin{proof}[Proof of \cref{conj:bombieri-sato-tate-antepenultimate}]
Combining \cref{prop:almost-bombieri,prop:bombieri-conversion}, and choosing $\eps, \eps'$ small enough depending on $\theta < 1/\max(2,\ell-1)$, we immediately obtain the result.
\end{proof}

\section*{Acknowledgements}

The authors wish to thank Professor Ken Ono and Professor Jesse Thorner for their guidance and suggestions.
We are grateful for the support of Emory University, the Asa Griggs Candler Fund, the NSA (grant H98230-19-1-0013), the NSF (grants 1557960 and 1849959), and the Spirit of Ramanujan Talent Initiative.

\appendix
\section{Minorization of indicator functions}\label{app:minorization}

\begin{lemma}\label{lem:Sym-8-minorization}
Let $[\alpha,\beta]\subseteq[-1,1]$.
We consider the following polynomials:
\begin{enumerate}
    \item $f_1(x):=-(x-\alpha)(x-\beta)[(x-x_1)(x-x_2)(x-x_3)]^2$,
    \item $f_2(x):=-(x-\alpha)(x-\beta)(1-x^2)[(x-x_1)(x-x_2)]^2$.
\end{enumerate}
If for some choice of $x_1,x_2,x_3\in[-1,1]$ we have $\int_{-1}^1f_i(x)\mu_{ST}(dx)>0$ for some $i\in\{1,2\}$, then $[\alpha,\beta]$ is $\Sym^8$-minorizable.
\end{lemma}
\begin{proof}
This is clear from the definition of $\Sym^8$-minorization.
\end{proof}

\begin{example}\label{ex:explicit-gap}
Let $I = [-1,-5/6]$. Then $\mu_{ST}(I) = 0.0398$ and $I$ can be $\Sym^8$-minorized by the polynomial $f(x) = (x-1)(x+5/6)(x+0.4)^2(x-0.16)^2(x-0.68)^2$
with corresponding $b_0 = 0.001017$. To achieve bounded gaps, by \cref{prop:maynard-primes-in-interval}, we need
$b_0M_k\theta/2 > 1$, i.e., $M_k>13766$. By the bound on $M_k$ given in \cite[Proposition~4.5]{T14}, we want $k\ge 213$ such that
\[\log k - 2\log\log k -2 > 13766,\]
so it suffices to pick $k =\lceil e^{13787.1}\rceil$ and take $\mc{H}$ to be the first $k$ prime numbers greater than $k$. By \cite{Du99}, for $n\geq 6$, the $n$th prime number satisfies the bound
\[ n(\log n + \log \log n-1)\leq p_n\leq n(\log n + \log \log n).\]
In particular, this shows that the number of primes $\leq k$ is at most 
\[\frac{k}{\log k-\log\log k-1} \leq \frac{k}{13776}.\] 
Therefore, the largest number in $\mathcal{H}$ is at most $p_{k+k/13776} \leq   10^{5991.81}$, hence we have
\[\liminf_{n\rightarrow \infty} (p_{I,n+1} - p_{I,n}) \leq \sup_{x,y\in \mathcal{H}}|x-y| \leq 10^{5992}.\]
\end{example}
\begin{lemma}\label{lem:half-of-intervals}
If we sample endpoints of an interval $I\subseteq [-1, 1]$ according to the Sato--Tate measure, then with at least a $50.74\%$ chance, $I$ can be $\Sym^8$-minorized. 
\end{lemma}
\begin{proof}
We use a brute-force computer program that provides a lower bound on the proportion of intervals $[\alpha,\beta]$ which satisfy the hypothesis of \cref{lem:Sym-8-minorization}.
The idea behind our program is to use the two forms of polynomials described in \cref{lem:Sym-8-minorization} as candidate minorizations of $I$. 
Our implementation is as follows, with $\eta_1 = 0.01$ and $\eta_2 = 0.0025$.
The notation $var\leftarrow val$ means that the value $val$ is assigned to the variable $var$.
\begin{enumerate}
    \item[1.] Initialize $S = 0$ and $\alpha=\beta=x_1=x_2=x_3 = -1$.
    \item[2.] Augment each parameter $x_1,x_2$, and $x_3$ by $\eta_1$ in nested loops from $-1$ to $1$, until one of the following cases happens:
    \begin{itemize}
        \item[a.] All $x_i$ reach $1$. If $\beta<1$, set $\beta \leftarrow \beta + \eta_2$, reset $x_i \leftarrow -1$ for all $i$, and repeat step 2. If $\beta=1$, then set $\alpha \leftarrow \alpha + \eta_2$ and reset $x_i \leftarrow -1$ for all $i$, and then set $\beta \leftarrow \alpha$. Then repeat step 2.
        
        \item[b.] The assumptions of \cref{lem:Sym-8-minorization} are satisfied for the current values of $\alpha,\beta,x_1,x_2$, and $x_3$. Accordingly, increment $S$ by the quantity $\mu_{ST}([\alpha - \eta_2,\alpha])\mu_{ST}([\beta,1])$ if $\alpha \neq -1$. Then increment $\alpha \leftarrow \alpha + \eta_2$ and reset $x_1 ,x_2,x_3 \leftarrow -1$ and $\beta \leftarrow \alpha$.
        Then repeat step 2.
    \end{itemize}
    \item[3.] When $\alpha$ reaches $1$, return the final value of $S$.
\end{enumerate}
Here, $2S$ is a lower bound for the proportion of all closed subintervals $I\subseteq [-1,1]$ in Sato--Tate measure that are $\Sym^8$-minorizable. (As we only consider the case when $\alpha < \beta$; such a sample space has Sato--Tate measure $1/2$.)
\end{proof}

\begin{lemma}\label{lem:minorization-big-measure}
If $I=[\alpha,\beta]\subseteq [-1,1]$ has Sato--Tate measure $\mu_{ST}(I) \geq 0.36$, then $I$ is $\Sym^8$-minorizable.
\end{lemma}
\begin{proof}
Modify the algorithm in the proof of \cref{lem:half-of-intervals} as follows:
\begin{enumerate}
    \item[1.] Each time that the assumptions of \cref{lem:Sym-8-minorization} are satisfied for the current values of $\alpha,\beta,x_1,x_2$, and $x_3$, we replace $S$ by the quantity $\max \{S, \mu_{ST}([\alpha-\eta_2,\beta])\}$ if $\alpha\neq -1$.
    \item[2.] Return the final value of $S$. 
\end{enumerate}
Here, $S$ is the measure of the smallest interval (in our search space) which is $\Sym^8$-minorizable.
\end{proof}

\bibliographystyle{plain}
\bibliography{main}

\end{document}